\documentclass{elsarticle}%
\usepackage{amsfonts}
\usepackage{amsthm}

\usepackage{amsmath}
\usepackage{amssymb}
\usepackage{bbold}
\usepackage{mathtools}
\usepackage{xcolor}
\usepackage{soul}
\usepackage{mathrsfs}
\usepackage{tabularx}
\usepackage[top=1in, bottom=1in, left=1in, right=1in]{geometry}
\usepackage{graphicx}%
\providecommand{\U}[1]{\protect\rule{.1in}{.1in}}

\newtheorem{theorem}{Theorem}[section]
\newtheorem{lemma}{Lemma}[section]

\theoremstyle{definition}
\newtheorem{definition}{Definition}[section]
\theoremstyle{remark}

\newtheorem{example}{Example}[section]

\numberwithin{equation}{section}
\begin{document}
    \begin{frontmatter}

		\title{Approximate controllability of semilinear impulsive evolution equations}


		
		\author[1]{Javad A. Asadzade}
		\ead{javad.asadzade@emu.edu.tr}
            \author[1,2]{Nazim I. Mahmudov}
		\ead{nazim.mahmudov@emu.edu.tr}

  \address[1]{Department of Mathematics, Eastern Mediterranean University, Mersin 10, 99628, T.R. North Cyprus, Turkey}
	\address[2]{Research Center of Econophysics, Azerbaijan State University of Economics (UNEC), Istiqlaliyyat Str. 6, Baku, 1001,  Azerbaijan}
	
		

		\begin{abstract}
Several dynamical systems in fields such as engineering, chemistry, biology, and physics  show impulsive behavior by reason of unexpected changes at specific times.
These behaviors are described by differential systems under impulse effects. The current paper examines approximate controllability for semi-linear impulsive differential
  and neutral differential equations in Hilbert spaces.
By applying a fixed-point method and semigroup theory, a new sufficient condition is provided for the (\(\mathcal{A}\)-controllability) approximate controllability of neutral and impulsive differential equations (IDEs). To demonstrate the value of the suggested consequences, three examples are presented,
offering improvements over some recent findings.

		\end{abstract}


	\end{frontmatter}

	\section{Introduction}
	
	\label{Sec:intro}

Impulsive behavior is a feature of many evolutionary processes, in which systems experience sudden, major disruptions at specific times. They effectively account for the impact of these sudden events on system states and provide an accurate representation of the dynamic behavior and properties of many processes. Numerous fields, including as computer science, genetics, population studies, artificial intelligence, neural networks, robotics, telecommunications, and biological systems, commonly use these kinds of impulsive dynamic systems. For a more thorough examination of impulsive systems, see monograph \cite{laks}.

In dynamic control systems, controllability is a basic property. When a system has an appropriate set of control functions, it may transition between any begining state and any target end state
in the space of its state in a finite period.

Recently, interest has surged in the \(\mathcal{A}\)-controllability of IDEs,
in which the state is influenced by impulses at a finite time intervals.
This topic has gained traction, with an expanding body of literature (see
\cite{laks}, \cite{pandit}, \cite{guan1}-\cite{muni}). Notable contributions
have been made by researchers such as George
et al.  \cite{geo},Benzaid and Sznaier \cite{ben},Muni and George \cite{muni}, Xie and Wang
\cite{xie2}, Guan et al.  \cite{guan1}, \cite{guan2},  Han et al. \cite{han},  Zhao and Sun \cite{zhao1}, \cite{zhao},
 among others.

In this paper, we analyze a criterion for \(\mathcal{A}\)-controllability of systems modeled by linear IDEs in abstract spaces.  It is assumed that the generator of a $C_0$-semigroup is the operator \( A \) influencing the state. A system is said to be \(\mathcal{A}\)-controllable if it can be moved from any beginning condition to a state that is close to a desired one.

We establish a criterion for the $\mathcal{A}$-controllability of linear IDEs by framing the problem as the limit of optimal control problems and redefining it in terms of the convergence of resolvent operators. The $\mathcal{A}$-controllability of a variety of  semilinear IDEs has been extensively studied using the so called resolvent condition, that is easy to apply. The analysis becomes more complex when impulses are introduced, even if this condition corresponds to $\mathcal{A}$-controllability in situations without impulses.

The necessary idea in the design and study of control systems is controllability.  Many different approaches have been used to explore the \(\mathcal{A}\)-controllability deterministic/stochastic differential systems in infinite-dimensional environments. There are two main types of controllability that are extensively applicable in the context of infinite dimensions: exact controllability and $\mathcal{A}$-controllability. A system can be directed into an arbitrarily small region of the ultimate state if it is $\mathcal{A}$-controllable, whereas it can reach a specific final state within a particular timeframe if it is exactly controllable. More prevalent and more appropriate for real-world uses are systems with $\mathcal{A}$-controllability. Studying approximation controllability in infinite-dimensional control systems is therefore essential.

During the past few years, necessary advancements have been made regarding the $\mathcal{A}$-controllability of deterministic and stochastic impulsive systems (see, for instance, \cite{14},\cite{23},  \cite{54} , \cite{54a}, \cite{54b}, etc.). These studies have established sufficient conditions for the $\mathcal{A}$-controllability of semilinear systems by employing the resolvent operator condition introduced in \cite{basmah}, \cite{mah}, particularly when the corresponding linear system is $\mathcal{A}$-controllable.

The resolvent condition is straightforward to apply and has been widely used in research on the $\mathcal{A}$-controllability of various semilinear IDEs. Without the presence of impulses, this condition aligns with the $\mathcal{A}$-controllability of the linear component of the corresponding semilinear evolution control system (see \cite{basmah}, \cite{mah}). However, the inclusion of impulses adds significant complexity to the analysis.  To date, there has been no study on the $\mathcal{A}$-controllability of controlled semilinear systems with control incorporated into the impulses. This paper marks the first effort to tackle this issue for semilinear deterministic systems. In this work, we investigate the $\mathcal{A}$-controllability of the following semilinear IDEs:
\begin{equation}\label{eq1}
	\begin{cases}
		\xi^{\prime}(t)=A\xi(t)+\Omega u(t)+\kappa(t,\xi(t)), & t\in\mathscr{I}=[0,b]\setminus\{t_{1},\dots,t_{p}\},\\
		\Delta \xi(t_{k+1})=B_{k+1}\xi(t_{k+1})+D_{k+1}v_{k+1},& k=0,\dots,p-1,\\
		\xi(0)=\xi_{0}.
	\end{cases}
\end{equation}

Here \(\xi(\cdot) \in H \) is in a Hilbert space .
The control \( u(\cdot) \) is an element of \(L^2([0, b], U) \), where ( \(H,\|\xi\| = \sqrt{\langle \xi, \xi\rangle} \) ) and \(U \) are Hilbert spaces, and \(v_k \in U \) for \(k = 1, \ldots, p \).

\bigskip

In this setting, \(A \) acts as a generator of a $C_0$-semigroup \(S(t) \) of continuous linear operators in \(H \). The  operators are: \(\Omega \in L(U, H) \), \(B_k \in L(H, H) \), and \(D_k \in L(U, H) \).

\bigskip

At each point of discontinuity \( t_k \) (for \( k = 1, \ldots, p \) with \( 0 = t_0 < t_1 < t_2 < \cdots < t_n < t_{p+1} = b \)), the state variable undergoes a jump, defined by \( \Delta \xi(t_k) = \xi(t_k^+) - \xi(t_k^-) \). Here, \( \xi(t_k^\pm) = \lim_{h \to 0^\pm} \xi(t_k + h) \), assuming that \( \xi(t_k^-) = \xi(t_k) \).

\bigskip

For operator compositions, \( \prod_{j=1}^{k} A_j \) represents the sequence \( A_1, A_2, \ldots, A_k \), while for \( j = k+1 \) to \( k \), \( \prod_{j=k+1}^{k} A_j = 1 \). Likewise, \( \prod_{j=k}^{1} A_j \) refers to the sequence \( A_k, A_{k-1}, \ldots, A_1 \), and \( \prod_{j=k}^{k+1} A_j = 1 \).

\bigskip

It is important to highlight that when \( H \) is an infinite-dimensional Hilbert space, under some natural conditions, the following linear IDEs:

\begin{align}\label{eq18}
\begin{cases}
    \xi^{\prime}(t) = A\xi(t) + \Omega u(t), & t \in \mathscr{I} = [0, b] \setminus \{t_{1}, \dots, t_{p}\}, \\
    \Delta \xi(t_{k+1}) = B_{k+1}\xi(t_{k+1}) + D_{k+1}v_{k+1}, & k = 0, \dots, p-1, \\
    \xi(0) = \xi_{0}.
\end{cases}
\end{align}

is $\mathcal{A}$-controllable, as noted in \cite{1}. This aspect is crucial for our article.
\bigskip

$\mathcal{A}$-controllability of neutral impulsive systems addresses the ability to drive the state of systems with both neutral (i.e., dependent on both the state and its derivatives) and impulsive (i.e., experiencing sudden changes at certain moments) characteristics to within an arbitrarily close distance of a desired target state.

In neutral impulsive systems, the dynamics are influenced not only by the present state but also by its delayed effects or derivatives, which adds complexity to the control problem. The occurrence of impulses introduces further complexity as it creates discontinuities in the system's trajectory. To study $\mathcal{A}$-controllability in these systems, mathematical techniques such as fixed-point theory, semigroup theory, and resolvent operator methods are typically applied. These approaches allow for deriving sufficient conditions that ensure that, despite delays and impulses, the system's state can be steered as close as desired to any given target.

Furthermore, we examine the $\mathcal{A}$-controllability of neutral IDEs, expressed in the following form:

\begin{equation}\label{eq188}
	\begin{cases}
		\frac{d}{dt}[\xi(t)+\sigma(t,\xi(t))]=A[\xi(t)+\sigma(t,\xi(t))]+\Omega u(t)+\kappa(t,\xi(t)), & t\in\mathscr{I}=[0,b]\setminus\{t_{1},\dots,t_{p}\},\\
		\Delta \xi(t_{k+1})=B_{k+1}\xi(t_{k+1})+D_{k+1}v_{k+1},& k=0,\dots,p-1,\\
		\xi(t)=\varphi(t),& t\in [-\tau, 0].
	\end{cases}
\end{equation}

In summary, our manuscript is structured as follows:
\bigskip

In Section 2, essential definitions, hypothesis, and theorems that underpin our main results are provided. Following this, Section 3 focuses on establishing the existence of solutions and the $\mathcal{A}$-controllability of semi-limear IDEs. Then, Section 4 extends these findings to neutral impulsive systems, presenting analogous results. Finally, in the concluding sections, we explore applications related to impulsive wave and heat equations, and we illustrate the solution of an impulsive semilinear equation in finite-dimensional spaces.
 \section{Theoretical background}
Let the value of \eqref{eq1} at the terminal time \( b \) be represented by \( \xi_b(\xi_0; u) \), which corresponds to the control \( u \) and the initial state \( \xi_0 \). We define the set

\[
\mathcal{R}(b, \xi_0) = \{ \xi_b(\xi_0; u)(0) : u(\cdot) \in L^2([0,b], U) \},
\]
which is referred to as the reachable set of \eqref{eq1} at final time \( b \). The closure of this set in the space \( H \) is indicated  by \(\overline{\mathcal{R}(b, \xi_0)} \).

\begin{definition}
We say that \eqref{eq1} is $\mathcal{A}$-controllable on \( [0,b] \) if \( \overline{\mathcal{R}(b, \xi_0)} = H \).
\end{definition}

For the sake of simplicity, let us define the operators as follows:
\begin{align*}
    \Gamma^{b}_{t_{p}}&=\int_{t_{p}}^{b}S(b-s)\Omega\Omega^{*}S^{*}(b-s)ds,\quad \tilde{\Gamma}^{b}_{t_{p}}=S(b-t_{p})D_{p}D^{*}_{p}S^{*}(b-t_{p}),\end{align*}
    \begin{align*}
            \Theta^{t_{p}}_{0}&=S(b-t_{p})\sum_{i=1}^{p}\prod_{j=p}^{i+1}(\mathcal{I}+B_{j})S(t_{j}-t_{j-1})
	(\mathcal{I}+B_{i})\int_{t_{i-1}}^{t_{i}}S(t_{i}-s)\Omega\Omega^{*}S^{*}(t_{k}-s)ds\\
 &\times (\mathcal{I}+B^{*}_{i})\prod_{k=i+1}^{p}S^{*}(t_{k}-t_{k-1})(\mathcal{I}+B^{*}_{k})S^{*}(b-t_{p}),
\end{align*}
\begin{align*}
    \tilde{\Theta}^{t_{p}}_{0}&=S(b-t_{p})\sum_{i=2}^{p}\prod_{j=p}^{i}(\mathcal{I}+B_{j})S(t_{j}-t_{j-1})D_{i-1}D^{*}_{i-1}\prod_{k=i}^{p}S^{*}(t_{k}-t_{k-1})(\mathcal{I}+B^{*}_{k})S^{*}(b-t_{p}).
\end{align*}
$(A_{0})$ $\alpha\Big(\alpha I+\Gamma^{b}_{t_{p}}+ \tilde{\Gamma}^{b}_{t_{p}}+\Theta^{t_{p}}_{0}+\tilde{\Theta}^{t_{p}}_{0}\Big)^{-1}\to 0$ as $\alpha\to 0^{+}$ in the strong operator topology.
\bigskip

Assumption $(A_{0})$ is equivalent to  the $\mathcal{A}$-controllablility of  \eqref{eq18}  on $[0,b]$, see \cite{1} (Theorem 13).
\bigskip

To simplify our discussion, we will now adopt the following notation:
\begin{align*}
    &K=\Vert B\Vert, \quad M=\max\Big\{\Vert S(t)\Vert : 0\leq t\leq b\Big\},\quad C=\max \Big\{ \Vert B_{i}\Vert,\quad for \quad i=1,\dots,p\Big\}, \\
    &\Vert \lambda_{i}\Vert =\int_{0}^{b}\vert \lambda_{i}(s)\vert ds,\quad \tilde{M}=\sum_{r=1}^{p+1} M^{r},\quad N=\max\Big\{ \tilde{M}, M^{p+1}(1+C)^{p}, M+\sum_{r=1}^{p-1}M^{r+2}(1+C)^{r+2}\Big\},\\
    &k=\max\{1, MN, MNb, NKb\},\quad a_{m}=3kKN^{2}\lambda_{m},\quad b_{m}=3N\lambda_{m},\quad c_{m}=\max\{a_{m},b_{m}\},\\
    &d_{1}=3kKN\Big[\Vert h\Vert+N\Vert \xi_{0}\Vert\Big],\quad  d_{2}=3N\Vert \xi_{0}\Vert+3NDV,\quad d=\max\{d_{1},d_{2}\},\\
    &D=\max \Big\{ \Vert D_{i}\Vert,\quad for \quad i=1,\dots,p\Big\},\quad V=\max \Big\{ \Vert v_{i}\Vert,\quad for \quad i=1,\dots,p\Big\}
\end{align*}

\bigskip
  We introduce the following hypothesis:
\bigskip

 \textbf{ $(A_{1})$:}
	\(A: D(A)\subset H\to H\) generates a semigroup  \(S(t), t > 0\) on \(H\), which is compact.
\bigskip

\textbf{$(A_{2})$:}
	The function \( \kappa: \mathscr{I} \times H \to H \) is continuous, and $\exists$ $\lambda_{m}(\cdot)\in L^{1}(\mathscr{I},R^{+})$ and $\varphi_{m}(\cdot)\in L^{1}(H,R^{+}), m=1,\dots, q$, such that
	\bigskip

 \begin{align*}
     \Vert \kappa(t,\xi(t))\Vert\leq \sum_{m=1}^{q}\lambda_{m}(t)\varphi_{m}(\xi)\quad \forall (t,\xi)\in \mathscr{I}\times H.
 \end{align*}

\bigskip

\textbf{$(A_{3})$:} For every \(\alpha> 0\),
\[
\limsup_{r \to \infty} \left( r - \sum_{m=1}^{q} \frac{c_{m}}{\alpha}\sup \left\{ \varphi_{m}(\xi) : \|\xi\| \leq r \right\} \right)=\infty.
\]

\bigskip

\textbf{$(A_{4})$:} The function \( \kappa : \mathscr{I} \times H \to H \) is uniformly bounded and continuous, meaning that there is a  \( N_1 > 0 \) such that

\[
\| \kappa(t,\xi) \| \leq N_1 \quad \text{for all} \ (t,\xi) \in \mathscr{I} \times H.
\]

\bigskip

In the following paragraph, we will indicate that system \eqref{eq1} is $\mathcal{A}$-controllable if, for every \( \alpha > 0 \), there is a \( \xi(\cdot) \in PC([0, b], H) \) so that

	\begin{align}\label{eq2}
	\xi(t)=
		\begin{cases}
			S(t)\xi(0)+\int_{0}^{t} S(t-s)\big[\Omega u(s)+\kappa(s,\xi(s))\big]ds,\quad  0\leq t\leq t_{1},\\
			\\
		S(t-t_{k})\xi(t^{+}_{k})+\int_{t_{k}}^{t} S(t-s)\big[\Omega u(s)+\kappa(s,\xi(s))\big]ds,\quad t_{k}<t\leq t_{k+1},\quad k=1,2,\dots, p,\\
		\end{cases}
	\end{align}
where
\begin{align}\label{eq3}
	\xi(t^{+}_{k})=&\prod_{j=k}^{1}(\mathcal{I}+B_{j})S(t_{j}-t_{j-1})\xi_{0}\nonumber\\
	+&\sum_{i=1}^{k}\prod_{j=k}^{i+1}(\mathcal{I}+B_{j})S(t_{j}-t_{j-1})
	(\mathcal{I}+B_{i})\int_{t_{i-1}}^{t_{i}}S(t_{i}-s)\Omega u(s)ds\nonumber\\
 +&\sum_{i=1}^{k}\prod_{j=k}^{i+1}(\mathcal{I}+B_{j})S(t_{j}-t_{j-1})
	(\mathcal{I}+B_{i})\int_{t_{i-1}}^{t_{i}}S(t_{i}-s)\kappa(s,\xi(s))ds\\
	+&\sum_{i=2}^{k}\prod_{j=k}^{i}(\mathcal{I}+B_{j}) S(t_{j}-t_{j-1}) D_{i-1}v_{i-1}+D_{k}v_{k},\nonumber
\end{align}

\begin{align}\label{l1}
    u_{\alpha}(s)&=\bigg( \sum_{k=1}^{p}\Omega^{*}S^{*}(t_{k}-s)\prod_{i=k+1}^{p}S^{*}(t_{i}-t_{i-1})S^{*}(b-t_{p})\chi_{(t_{k-1},t_{k})}+\Omega^{*}S^{*}(b-s)\chi_{(t_{p},b)}\bigg)\tilde{\varphi}_{\alpha},
\end{align}
\begin{align}\label{l111}
    v^{\alpha}_{p}=D^{*}_{p}S^{*}(b-t_{p})\tilde{\varphi}_{\alpha},\quad  v^{\alpha}_{k}=D^{*}_{k}\prod_{i=k}^{p}S^{*}(t_{i} -t_{i-1})(I +B^{*}_i)S^{*}(b-t_p)\tilde{\varphi}_{\alpha},\quad k=1,\dots, p-1,
\end{align}
\begin{align*}
\tilde{\varphi}_{\alpha}(\xi(\cdot))=&\Big(\alpha\mathcal{I}+\Theta^{t_{p}}_{0}+\Gamma^{b}_{t_{p}}+\tilde{\Theta}^{t_{p}}_{0}+\tilde{\Gamma}^{b}_{t_{p}}\Big)^{-1}\bigg(h-S(b-t_{p})\prod_{j=p}^{1}(\mathcal{I}+B_{j})S(t_{j}-t_{j-1})\xi_{0}\\
    -&S(b-t_{p})\sum_{i=1}^{p}\prod_{j=p}^{i+1}(\mathcal{I}+B_{j})S(t_{j}-t_{j-1})
	(\mathcal{I}+B_{i})\int_{t_{i-1}}^{t_{i}}S(t_{i}-s)\kappa(s,\xi(s))ds\\
 -&\int_{t_{p}}^{b} S(b-s)\kappa(s,\xi(s))ds\bigg).
\end{align*}

For piecewise continuous functions, which are functions that may have somer discontinuities on an interval, the Ascoli–Arzelà theorem can be adapted, but certain conditions are required to account for these discontinuities.
We present an extended version of the Ascoli–Arzelà theorem, as demonstrated by W. Wei, X. Xiang, and Y. Peng in their work on \( PC(I,  \mathfrak{X}) \) in \cite{9}, where \( \mathfrak{X} \) denotes a Banach space. This extended version generalizes the classical result to the space of piecewise continuous functions, providing conditions under which a set in \( PC(I, \mathfrak{X}) \) is relatively compact.

\begin{theorem} \textbf{(Ascoli–Arzelà theorem  }\label{Ascoli-Arsela}
Assume \( \mathfrak{W} \subseteq PC(I, \mathfrak{X}) \). If the following conditions are held:
\begin{enumerate}
    \item Uniform Boundedness: The set \( \mathfrak{W} \) is a uniformly bounded subset of \( PC(I, \mathfrak{X}) \).
    \item Equicontinuity on subintervals: The set \( \mathfrak{W} \) is equicontinuous in \( I_i = (t_i, t_{i+1}) \), where \( i = 0, 1, 2, \dots, n \), with \( t_0 = 0 \) and \( t_{n+1} = T \).
    \item Control at Discontinuities: \( \mathfrak{W}(t) = \{\xi(t) \mid \xi \in \mathfrak{W}, \ t \in I \setminus D \} \), \( \mathfrak{W}(t_i + 0) = \{\xi(t_i + 0) \mid \xi \in \mathfrak{W}\} \), and \( \mathfrak{W}(t_i - 0) = \{\xi(t_i - 0) \mid \xi \in \mathfrak{W}\} \) are relatively compact subsets of \( \mathfrak{X} \).
\end{enumerate}
Then \( \mathfrak{W} \subseteq PC(I, \mathfrak{X}) \) is a relatively compact.
\end{theorem}

This result is significant in applications involving piecewise continuous functions, as it allows for compactness considerations in the presence of discontinuities. It is particularly useful in the analysis of impulsive systems in control theory and differential equations, where piecewise continuous functions model sudden state changes.

\section{$\mathcal{A}$-controllability of IDEs}

The Schauder Fixed-Point Theorem (SFPT) is a foundational result in functional analysis that provides conditions under which a function has at least one fixed point. It is particularly useful in proving the existence of solutions to various types of issues in analysis and differential equations.

In the following theorem to show existence of solution we apply SFPT.

\begin{theorem}
Under assumptions $A_{1}-A_{3}$ the system \eqref{eq1} has a solution on \( \mathscr{I} \) for every \( 0 < \alpha < 1 \); that is, $F_{\alpha}$ has a fixed point.
\end{theorem}
\begin{proof}
The major purpose of this section is to establish the requirements for the solvability of system \eqref{eq2} and \eqref{l1} for \(\alpha >0\).  In the  space \(PC(\mathscr{I}, H)\), we consider the set
\[
B_{r(\alpha)} = \{ \xi(\cdot) \in PC(\mathscr{I}, H) \mid \xi(0) = \xi_0, \|\xi\| \leq r(\alpha) \},
\]
where \(r(\alpha)>0\) is a constant.

We introduce an operator \(F_\alpha, \alpha > 0\) on \(PC(\mathscr{I}, H)\) in the following way
\[
F_\alpha(\xi) = z,
\]
such that
\begin{align}
	z(t)=
		\begin{cases}
			S(t)\xi(0)+\int_{0}^{t} S(t-s)\big[\Omega v(s)+\kappa(s,\xi(s))\big]ds,\quad  0\leq t\leq t_{1},\\
			\\
		S(t-t_{k})\xi(t^{+}_{k})+\int_{t_{k}}^{t} S(t-s)\big[\Omega v(s)+\kappa(s,\xi(s))\big]ds,\quad t_{k}<t\leq t_{k+1},\quad k=1,2,\dots, p,\\
		\end{cases}
	\end{align}
where
\begin{align}\label{eq3654}
	\xi(t^{+}_{k})=&\prod_{j=k}^{1}(\mathcal{I}+B_{j})S(t_{j}-t_{j-1})\xi_{0}\nonumber\\
	+&\sum_{i=1}^{k}\prod_{j=k}^{i+1}(\mathcal{I}+B_{j})S(t_{j}-t_{j-1})
	(\mathcal{I}+B_{i})\int_{t_{i-1}}^{t_{i}}S(t_{i}-s)\Omega u_{\alpha}(s)ds\nonumber\\
 +&\sum_{i=1}^{k}\prod_{j=k}^{i+1}(\mathcal{I}+B_{j})S(t_{j}-t_{j-1})
	(\mathcal{I}+B_{i})\int_{t_{i-1}}^{t_{i}}S(t_{i}-s)\kappa(s,\xi(s))ds\\
	+&\sum_{i=2}^{k}\prod_{j=k}^{i}(\mathcal{I}+B_{j}) S(t_{j}-t_{j-1}) D_{i-1}v_{i-1}+D_{k}v_{k}.\nonumber
\end{align}
\begin{align}\label{l119}
    v(s)&=\bigg( \sum_{k=1}^{p}\Omega^{*}S^{*}(t_{k}-s)\prod_{i=k+1}^{p}S^{*}(t_{i}-t_{i-1})S^{*}(b-t_{p})\chi_{(t_{k-1},t_{k})}+\Omega^{*}S^{*}(b-s)\chi_{(t_{p},b)}\bigg)\tilde{\varphi}_{\alpha},
\end{align}
where
\begin{align*}
\tilde{\varphi}_{\alpha}=&\Big(\alpha\mathcal{I}+\Theta^{t_{p}}_{0}+\Gamma^{b}_{t_{p}}+\tilde{\Theta}^{t_{p}}_{0}+\tilde{\Gamma}^{b}_{t_{p}}\Big)^{-1}\bigg(h-S(b-t_{p})\prod_{j=p}^{1}(\mathcal{I}+B_{j})S(t_{j}-t_{j-1})\xi_{0}\\
    -&S(b-t_{p})\sum_{i=1}^{p}\prod_{j=p}^{i+1}(\mathcal{I}+B_{j})S(t_{j}-t_{j-1})
	(\mathcal{I}+B_{i})\int_{t_{i-1}}^{t_{i}}S(t_{i}-s)\kappa(s,\xi(s))ds\\
 -&\int_{t_{p}}^{b} S(b-s)\kappa(s,\xi(s))ds\bigg)
\end{align*}
\bigskip
To enhance clarity, the proof of the theorem will be broken down into two steps due to its length and complexity.

\bigskip

\textbf{Step 1.}  For any \(\alpha > 0\) $\exists$  \(r(\alpha)>0\) constant such that the mapping \(F_\alpha\) satisfies:
$ F_\alpha : B_{r(\alpha)} \rightarrow B_{r(\alpha)}. $  Let
\begin{align*}
    \Phi_{m}(r)=\sup \Big\{ \varphi_{m}(\xi): \Vert y
    \Vert \leq r,  y\in H\Big\}.
\end{align*}
By assumption $(A_{3})$, for any \(\alpha > 0\) $\exists$  \(r(\alpha)>0\) such that
\begin{align*}
    \frac{d}{\alpha}+\sum_{m=1}^{q}\frac{c_{m}}{\alpha}\Phi_{m}(r(\alpha))\leq r(\alpha).
\end{align*}
If \( \xi(\cdot) \in B_{r(\alpha)} \), then we obtain
\begin{align*}
    \Vert v(s)\Vert &\leq\bigg\Vert \sum_{k=1}^{p}\Omega^{*}S^{*}(t_{k}-s)\prod_{i=k+1}^{p}S^{*}(t_{i}-t_{i-1})S^{*}(b-t_{p})\chi_{(t_{k-1},t_{k})}+\Omega^{*}S^{*}(b-s)\chi_{(t_{p},b)}\bigg\Vert \Vert\tilde{\varphi}_{\alpha}\Vert\\
  &\leq  \Vert \Omega^{*}\Vert\sum_{r=1}^{p+1} M^{r}\Vert\tilde{\varphi}_{\alpha}\Vert\leq \frac{1}{\alpha} K\tilde{M}\bigg\Vert h-S(b-t_{p})\prod_{j=p}^{1}(\mathcal{I}+B_{j})S(t_{j}-t_{j-1})\xi_{0}\\
  -&S(b-t_{p})\sum_{i=1}^{p}\prod_{j=p}^{i+1}(\mathcal{I}+B_{j})S(t_{j}-t_{j-1})
	(\mathcal{I}+B_{i})\int_{t_{i-1}}^{t_{i}}S(t_{i}-s)\kappa(s,\xi(s))ds\\
    -&\int_{t_{p}}^{b} S(b-s)\kappa(s,\xi(s))ds\bigg\Vert\leq \frac{1}{\alpha} K\tilde{M}\bigg\Vert h-S(b-t_{p})\prod_{j=p}^{1}(\mathcal{I}+B_{j})S(t_{j}-t_{j-1})\xi_{0}\\
    -&S(b-t_{p})\sum_{i=1}^{p}\prod_{j=p}^{i+1}(\mathcal{I}+B_{j})S(t_{j}-t_{j-1})
	(\mathcal{I}+B_{i})\int_{t_{i-1}}^{t_{i}}S(t_{i}-s)\sum_{m=1}^{q}\lambda_{m}(s)\varphi_{m}(\xi(s))ds\\
     -&\int_{t_{p}}^{b} S(b-s)\sum_{m=1}^{q}\lambda_{m}(s)\varphi_{m}(\xi(s))ds\bigg\Vert\leq \frac{1}{\alpha} K\tilde{M} \bigg[\Vert h\Vert+M^{p+1}(1+C)^{p} \Vert \xi_{0}\Vert\\
     &+\sum_{r=1}^{p-1}M^{r+2}(1+C)^{r+2}\int_{0}^{b}\sum_{m=1}^{q}\lambda_{m}(s)\varphi_{m}(\xi(s))ds+M\int_{0}^{b}\sum_{m=1}^{q}\lambda_{m}(s)\varphi_{m}(\xi(s))ds\bigg]\\
 &\leq \frac{1}{\alpha} K\tilde{M}\bigg[\Vert h\Vert+M^{p+1}(1+C)^{p} \Vert \xi_{0}\Vert+\Big(M+\sum_{r=1}^{p-1}M^{r+2}(1+C)^{r+2}\Big)\sum_{m=1}^{q} \Vert \lambda_{m}\Vert \Phi_{m}(r(\alpha))\bigg]\\
 &\leq \frac{1}{\alpha} K\tilde{M}\bigg[\Vert h\Vert+M^{p+1}(1+C)^{p} \Vert \xi_{0}\Vert\bigg]+\frac{1}{\alpha} K\tilde{M} \bigg[\Big(M+\sum_{r=1}^{p-1}M^{r+2}(1+C)^{r+2}\Big)\sum_{m=1}^{q} \Vert \lambda_{m}\Vert \Phi_{m}(r(\alpha))\bigg]\\
 &\leq \frac{1}{\alpha} KN\Big[\Vert h\Vert+N\Vert \xi_{0}\Vert\Big]+\frac{1}{\alpha} KN^{2}\sum_{m=1}^{q} \Vert \lambda_{m}\Vert \Phi_{m}(r(\alpha))\leq \frac{d}{3k\alpha}+\frac{1}{3k}\sum_{m=1}^{q}\frac{c_{m}}{\alpha}\Phi_{m}(r(\alpha))\\
 &=\frac{1}{3k}\bigg(\frac{d}{\alpha}+\sum_{m=1}^{q}\frac{c_{m}}{\alpha}\Phi_{m}(r(\alpha))\bigg)\leq \frac{r(\alpha)}{3k}.
\end{align*}
For $0\leq t\leq t_{1}$, we have

\begin{align*}
    \Vert z\Vert &\leq M\Vert \xi_{0}\Vert+MKb\Vert v\Vert+M\int_{0}^{t}\sum_{m=1}^{q}\lambda_{m}(s)\varphi_{m}(\xi(s))ds\\
    &\leq \frac{1}{3}\Big[ d+\sum_{m=1}^{q} c_{m}\Phi_{m}(r(\alpha))\Big]+k\Vert v\Vert\leq \frac{\alpha r(\alpha)}{3}+\frac{r(\alpha)}{3}\leq \frac{2r(\alpha)}{3}.
\end{align*}
For $ t_{k}<t\leq t_{k+1}, k=1,2,\dots, p$, we have
\begin{align*}
    \Vert z\Vert &\leq M^{k+1}(1+C)^{k}\Vert \xi_{0}\Vert+\sum_{r=1}^{k-1}M^{r+2}(1+C)^{r+2}Kb\Vert  v\Vert \\
    &+\sum_{r=1}^{k-1}M^{r+2}(1+C)^{r+2}\int_{0}^{b}\sum_{m=1}^{q}\lambda_{m}(s)\varphi_{m}(\xi(s))ds\\
    &+\sum_{r=1}^{k}M^{r}(1+C)^{r-1}DV+MKb\Vert v\Vert\\
    &+M\int_{0}^{b}\sum_{m=1}^{q}\lambda_{m}(s)\varphi_{m}(\xi(s))ds\\
    &\leq M^{k+1}(1+C)^{k}\Vert \xi_{0}\Vert+\sum_{r=1}^{k}M^{r}(1+C)^{r-1}DV\\
    &+\bigg(M+\sum_{r=1}^{k-1}M^{r+2}(1+C)^{r+2}\bigg)Kb\Vert  v\Vert \\
     &+\bigg(M+\sum_{r=1}^{k-1}M^{r+2}(1+C)^{r+2}\bigg)\sum_{m=1}^{q}\Vert \lambda_{m}\Vert\Phi_{m}(r(\alpha))\\
       &\leq N\Vert \xi_{0}\Vert+NDV+N\sum_{m=1}^{q}\Vert \lambda_{m}\Vert\Phi_{m}(r(\alpha))+NKb\Vert  v\Vert      \\
        &\leq \frac{1}{3}\Big[ d+\sum_{m=1}^{q} c_{m}\Phi_{m}(r(\alpha))\Big]+k\Vert v\Vert\leq \frac{\alpha r(\alpha)}{3}+\frac{r(\alpha)}{3}\leq \frac{2r(\alpha)}{3}.
\end{align*}

 So

 \begin{align*}
     \Vert F_{\alpha}(\xi)(t)\Vert=\Vert z(t)\Vert+\Vert v\Vert\leq r(\alpha).
 \end{align*}
 Then $F_{\alpha}$ maps $B_{r(\alpha)}$ into itself.
\bigskip

\textbf{Step 2.}
For any \(\alpha > 0\), the operator \(F_\alpha\) maps the set \(B_{r(\alpha)}\) into a subset of itself that is relatively compact. Additionally, \(F_\alpha\) possesses a fixed point within \(B_{r(\alpha)}\).

In accordance with the Ascoli–Arzelà theorem, it is necessary to show that

\bigskip

$(i)$  For  $\forall$ \(t \in \mathscr{I}\), the set \(\mathcal{V}(t) = \{(F_{\alpha} \xi)(t) : \xi(\cdot) \in B_{r(\alpha)}\}\) is relatively compact.

\bigskip

$(ii)$  The set \(\mathcal{V} = \{ (F_{\alpha}\xi)(\cdot) \mid \xi(\cdot) \in B_{r(\alpha)} \} \) is equicontinuous on \(\mathscr{I}\).
\bigskip

Let us prove part $(i)$. The case when \( t = 0 \) is straightforward, as \( \mathcal{V}(0) = \{\xi_0\} \). Now, let \( t \) be a fixed real number such that \( 0 < t \leq b \), and consider a real number \( \tau \) satisfying \( 0 < \tau< t \). Define

\begin{align}\label{mop1}
 (F^{\tau}_{\alpha}\xi)(t)=S(t)\xi_{0}+S(\tau)\mathcal{V}(t-\tau).
\end{align}
Since \( z(t - r) \) is bounded on \( B_{r(\alpha)} \) and \( S(t) \) is compact, the set
\[
\mathcal{V}_{\tau}(t) = \{ (F^{\tau}_{\alpha}\xi)(t) : \xi(\cdot) \in B_{r(\alpha)} \}
\]
is relatively compact in \( H \). This implies that there exists a finite set \( \{ y_i \mid 1 \leq i \leq n \} \) in \( H \) such that
\[
\mathcal{V}_{\tau}(t)  \subset \bigcup_{i=1}^n N\big(y_i, \frac{\varepsilon}{2}\big),
\]
where \( N\big(y_i, \frac{\varepsilon}{2}\big) \) represents an open ball in \( H \) with center \( \xi_i \) and radius \( \varepsilon/2 \). Additionally, for $0\leq t\leq t_{1}$, we have
\begin{align*}
\Vert (F_{\alpha}\xi)(t)- (F^{\tau}_{\alpha}\xi)(t)\Vert &= \bigg\Vert\int_{t-\tau}^t S(t - s) [\Omega v(s) + \kappa(s, \xi(s))] ds\bigg\Vert\\
&\leq MK\tau\Vert v\Vert +M\int_{t-\tau}^{t}\sum_{m=1}^{q}\lambda_{m}(s)\varphi_{m}(\xi(s))ds\\
&\leq MK\tau\frac{r(\alpha)}{3k} +M\int_{t-\tau}^{t}\sum_{m=1}^{q}\lambda_{m}(s)ds\Phi_{m}(r(\alpha))\leq \frac{\varepsilon}{2}.
\end{align*}

 Consider interval $(t_{1},t_{2}]$ , we define

 \begin{align*}
     \mathcal{V}(t_{1}+0)&=\mathcal{V}(t_{1}-0)+B_{1}\mathcal{V}(t_{1})+D_{1}v_{1}\\
     &=(I+B_{1})\mathcal{V}(t_{1})+D_{1}v_{1}.
 \end{align*}
Similarly, we have $ \mathcal{V}(t_{1}+0)$ is relatively compact. Let $\xi(t_{1})\equiv \xi_{1}$, then equation \eqref{mop1} reduces to
 \begin{align*}
   (F^{\tau}_{\alpha}\xi)(t)=S(t-t_{1})\xi_{1}+S(\tau)\mathcal{V}(t-\tau).
\end{align*}
 Furthermore,
\begin{align*}
\Vert (F_{\alpha}\xi)(t)- (F^{\tau}_{\alpha}\xi)(t)\Vert\leq MK\tau\frac{r(\alpha)}{3k} +M\int_{t-\tau}^{t}\sum_{m=1}^{q}\lambda_{m}(s)ds\Phi_{m}(r(\alpha))\leq \frac{\varepsilon}{2},
\end{align*}
  thus $\mathcal{V}(t)$ is relatively compact for $t\in (t_{1},t_{2}]$.
  \bigskip

 In general, given any $t_{k}\in\tilde{D}=\{t_{1},\dots, t_{p}\}$ for $k=1,\dots,p$,  we set
\begin{align*}
    \xi(t_{k}+0)=\xi(t_{k})
\end{align*}
 and
 \begin{align*}
     \mathcal{V}(t_{k}+0)&=\mathcal{V}(t_{k}-0)+B_{k}\mathcal{V}(t_{k})+D_{k}v_{k}=(I+B_{k})\mathcal{V}(t_{k})+D_{k}v_{k},\quad  k=1,\dots,p.
 \end{align*}
Such that , we know that $\mathcal{V}(t_{k}+0)$ is relatively compact and the
 associated $\mathcal{V}_{\tau}(t)$ over the interval $(t_{k},t_{k+1}]$ is given by
 \begin{align*}
   (F^{\tau}_{\alpha}\xi)(t)=S(t-t_{k})\xi_{k}+S(\tau)\mathcal{V}(t-\tau),\quad k=1,2,\dots, p .
\end{align*}
Similarly, for $ t_{k}<t\leq t_{k+1}, k=1,2,\dots, p$, we have
\begin{align*}
\Vert (F_{\alpha}\xi)(t)- (F^{\tau}_{\alpha}\xi)(t)\Vert\leq MK\tau\frac{r(\alpha)}{3k} +M\int_{t-\tau}^{t}\sum_{m=1}^{q}\lambda_{m}(s)ds\Phi_{m}(r(\alpha))\leq \frac{\varepsilon}{2}.
\end{align*}
 Consequently, we obtain
 \[
\mathcal{V}(t)  \subset \bigcup_{i=1}^n N\big(y_i, \varepsilon\big),
\]
Thus, for every \(t \in [0, b]\), \(V(t)\) is relatively compact in the Hilbert space \(H\).
\bigskip

To prove $(ii)$, we need to demonstrate that the set \(V = \{ (F_{\alpha}\xi)(\cdot) \mid \xi(\cdot) \in B_{r(\alpha)} \} \) is equicontinuous on \([0, b]\).  In fact, for $0<a_{1}<a_{2}\leq b$, we achive
\begin{align*}
    \Vert v(a_{2})-v(a_{1})\Vert&\leq \Bigg\Vert \sum_{k=1}^{p}\Omega^{*}\Big(S^{*}(t_{k}-a_{2})-S^{*}(t_{k}-a_{1})\Big)\prod_{i=k+1}^{p}S^{*}(t_{i}-t_{i-1})S^{*}(b-t_{p})\chi_{(t_{k-1},t_{k})}\\
    &+\Omega^{*}\Big(S^{*}(b-a_{2})-S^{*}(b-a_{1})\Big)\chi_{(t_{p},b)}\bigg\Vert\\
    &\times\frac{1}{\alpha} \bigg[\Vert h\Vert+M^{p+1}(1+C)^{p} \Vert \xi_{0}\Vert+\Big(M+\sum_{r=1}^{p-1}M^{r+2}(1+C)^{r+2}\Big)\sum_{m=1}^{q} \Vert \lambda_{m}\Vert \Phi_{m}(r(\alpha))\bigg].
\end{align*}
For $0<a_{1}<a_{2}\leq t_{1}$, we get
\begin{align}\label{1}
    \Vert z(a_{2})-z(a_{1})\Vert \leq &\Vert S(a_2) - S(a_1)\Vert \Vert \xi_0\Vert + KM \int_{a_{1}}^{a_{2}} \Vert v(s) \Vert ds\nonumber\\
    +& K \int_{0}^{a_1} \Vert S(a_2 - s) - S(a_1 - s)\Vert \Vert v(s) \Vert ds\nonumber\\
+& M \int_{a_{1}}^{a_{2}} \sum_{m=1}^{q} \lambda_m(s) \varphi_m(\xi(s)) \, ds\nonumber \\
+&  \int_{0}^{a_1} \Vert S(a_2 - s) - S(a_1 - s)\Vert \sum_{m=1}^{q} \lambda_m(s) \varphi_m(\xi(s)) \, ds\nonumber\\
\leq &\Vert S(a_2) - S(a_1)\Vert \Vert \xi_0\Vert + KM \int_{a_1}^{a_2} \Vert v(s) \Vert ds\\
+& K \int_{0}^{a_1} \Vert S(a_2 - s) - S(a_1 - s)\Vert \Vert v(s) \Vert ds\nonumber\\
+& M \sum_{m=1}^{q} \int_{a_1}^{a_2}  \lambda_m(s) ds \Phi_m(r(\alpha))\nonumber\\
+&  \sum_{m=1}^{q}\int_{0}^{a_1} \Vert S(a_2 - s) - S(a_1 - s)\Vert  \lambda_m(s)  ds \Phi_m(r(\alpha))\nonumber\\
=& I_1 + I_2 + I_3 + I_4 + I_5.\nonumber
\end{align}

For $t_{k}<a_{1}<a_{2} \leq t_{k+1}, \quad k=1,2,\dots,p$, we get
\begin{align}\label{2}
    \Vert z(a_{2})-z(a_{1})\Vert \leq &\Vert S(a_2) - S(a_1)\Vert \Vert \xi(t_{k}^{+})\Vert + KM \int_{a_1}^{a_2} \Vert v(s) \Vert ds + K \int_{t_{k}}^{a_1} \Vert S(a_2 - s) - S(a_1 - s)\Vert \Vert v(s) \Vert ds\nonumber\\
+& M \sum_{m=1}^{q} \int_{a_1}^{a_2}  \lambda_m(s) ds \Phi_m(r(\alpha)) +  \sum_{m=1}^{q}\int_{t_{k}}^{a_1} \Vert S(a_2 - s) - S(a_1 - s)\Vert  \lambda_m(s)  ds\Phi_m(r(\alpha))\nonumber\\
\leq &\Vert S(a_2) - S(a_1)\Vert\Big\{ (1+C)^{k} M^{k}\Vert \xi_{0}\Vert+ \sum_{r=1}^{k-1}M^{r+1}(1+C)^{r+2}Kb\Vert  v\Vert \nonumber\\
    +&\sum_{r=1}^{k-1}M^{r+1}(1+C)^{r+2}\sum_{m=1}^{q}\Vert\lambda_{m}\Vert\Phi_{m}(r(\alpha))+\sum_{r=1}^{k}M^{r-1}(1+C)^{r-1}DV\Big\}\\
+& KM \int_{a_1}^{a_2} \Vert v(s) \Vert ds + K \int_{t_{k}}^{a_1} \Vert S(a_2 - s) - S(a_1 - s)\Vert \Vert v(s) \Vert ds\nonumber\\
+& M \sum_{m=1}^{q} \int_{a_1}^{a_2}  \lambda_m(s) ds \Phi_m(r(\alpha)) +  \sum_{m=1}^{q}\int_{t_{k}}^{a_1} \Vert S(a_2 - s) - S(a_1 - s)\Vert  \lambda_m(s)  ds\Phi_m(r(\alpha))\nonumber\\
=&J_1 +J_2 +J_3 +J_4 +J_5.\nonumber
\end{align}



In equations (3.4) and (3.5), the right-hand sides are not influence on the selection of \( \xi(\cdot) \).      As \( a_2 - a_1 \to 0 \), both \( I_2 \) and \( I_4 \) (and similarly \( J_2 \) and \( J_4 \)) tend to zero. Since the semigroup \( S(\cdot) \) is compact, we can deduce that
\[
S(t_2 - s) - S(t_1 - s) \to 0 \quad \text{as} \quad a_2 - a_1 \to 0,
\]
for any \( t \) and \( s \) where \( t - s > 0 \). This implies that \( I_1 \to 0 \) and \( J_1 \to 0 \). Additionally, employing the Lebesgue dominated convergence theorem, we conclude that
\[
I_3 \to 0, \quad I_5 \to 0, \quad J_3 \to 0, \quad \text{and} \quad J_5 \to 0 \quad \text{as} \quad a_2 - a_1 \to 0,
\]
demonstrating that \( V \) is equicontinuous. Consequently, the operator \( F_\alpha Br(\alpha) \) is both equicontinuous and bounded. According to the Ascoli–Arzelà theorem, \( F_\alpha Br(\alpha) \) is relatively compact in \( PC(J, H) \). Furthermore, for every \( \alpha > 0 \), the operator \( F_\alpha \) is continuous on \( PC(J, H) \), making \( F_\alpha \) a compact, continuous operator on \( PC(J, H) \). By the SFPT, it follows that \( F_\alpha \) has a fixed point.

\end{proof}

Examine the subsequent linear system with $\kappa(\cdot)\in L^{1}(\mathscr{I},H)$.
\begin{align}\label{14}
    z(t,\xi_{0})&=\begin{cases}
			S(t)\xi(0)+\int_{0}^{t} S(t-s)\big[\Omega u_{\alpha}(s)+\kappa(s)\big]ds,\quad  0\leq t\leq t_{1},\\
			\\
		S(t-t_{k})\xi(t^{+}_{k})+\int_{t_{k}}^{t} S(t-s)\big[\Omega u_{\alpha}(s)+\kappa(s)\big]ds,\quad t_{k}<t\leq t_{k+1},\quad k=1,2,\dots, p,\\
		\end{cases}
	\end{align}
where
\begin{align*}
	\xi(t^{+}_{k})=&\prod_{j=k}^{1}(\mathcal{I}+B_{j})S(t_{j}-t_{j-1})\xi_{0}\nonumber\\
	+&\sum_{i=1}^{k}\prod_{j=k}^{i+1}(\mathcal{I}+B_{j})S(t_{j}-t_{j-1})
	(\mathcal{I}+B_{i})\int_{t_{i-1}}^{t_{i}}S(t_{i}-s)[\Omega u_{\alpha}(s)+\kappa(s)]ds\nonumber\\
	+&\sum_{i=2}^{k}\prod_{j=k}^{i}(\mathcal{I}+B_{j}) S(t_{j}-t_{j-1}) D_{i-1}v_{i-1}+D_{k}v_{k}.\nonumber
\end{align*}

\begin{lemma}
If
\begin{align*}
p=&h-S(b-t_{p})\prod_{j=p}^{1}(\mathcal{I}+B_{j})S(t_{j}-t_{j-1})\xi_{0}\\
    -&S(b-t_{p})\sum_{i=1}^{p}\prod_{j=p}^{i+1}(\mathcal{I}+B_{j})S(t_{j}-t_{j-1})
	(\mathcal{I}+B_{i})\int_{t_{i-1}}^{t_{i}}S(t_{i}-s)\kappa(s)ds\\
 -&\int_{t_{p}}^{b} S(b-s)\kappa(s)ds
\end{align*}
and if \( u_\alpha(\cdot) \in L^2(\mathscr{I}, U) \) is a control function defined by

    \begin{align}\label{17}
    u_{\alpha}(s)&=\bigg( \sum_{k=1}^{p}\Omega^{*}S^{*}(t_{k}-s)\prod_{i=k+1}^{p}S^{*}(t_{i}-t_{i-1})S^{*}(b-t_{p})\chi_{(t_{k-1},t_{k})}+\Omega^{*}S^{*}(b-s)\chi_{(t_{p},b)}\bigg)\tilde{\varphi}_{\alpha},
\end{align}
where
\begin{align*}
\tilde{\varphi}_{\alpha}=\Big(\alpha\mathcal{I}+\Theta^{t_{p}}_{0}+\Gamma^{b}_{t_{p}}+\tilde{\Theta}^{t_{p}}_{0}+\tilde{\Gamma}^{b}_{t_{p}}\Big)^{-1} p.
\end{align*}
Then
\begin{align}\label{38}
    z(b,\xi_{0}) -h
 = -\alpha\Big(\alpha\mathcal{I}+\Theta^{t_{p}}_{0}+\Gamma^{b}_{t_{p}}+\tilde{\Theta}^{t_{p}}_{0}+\tilde{\Gamma}^{b}_{t_{p}}\Big)^{-1} p
  \end{align}
  and
\begin{align*}
    z(t,\xi_{0})
      &=S(t-t_{k})\prod_{j=k}^{1}(\mathcal{I}+B_{j})S(t_{j}-t_{j-1})\xi_{0}\\
    &+S(t-t_{k})\sum_{i=1}^{k}\prod_{j=k}^{i+1}(\mathcal{I}+B_{j})S(t_{j}-t_{j-1})
	(\mathcal{I}+B_{i})\int_{t_{i-1}}^{t_{i}}S(t_{i}-s)\kappa(s)ds\\
 &+\int_{t_{k}}^{t} S(t-s)\kappa(s)ds+\Big(\Theta^{t_{p}}_{0}+\Gamma^{b}_{t_{p}}+\tilde{\Theta}^{t_{p}}_{0}+\tilde{\Gamma}^{b}_{t_{p}}\Big)S^{*}(b-t)\Big(\alpha\mathcal{I}+\Theta^{t_{p}}_{0}+\Gamma^{b}_{t_{p}}+\tilde{\Theta}^{t_{p}}_{0}+\tilde{\Gamma}^{b}_{t_{p}}\Big)^{-1}p.
  \end{align*}
\end{lemma}
\begin{proof}
   Replacing  \eqref{17} in  \eqref{14}, we acquire the following findings.

 \begin{align*}
   & z(t,\xi_{0})=S(t-t_{k})\prod_{j=k}^{1}(\mathcal{I}+B_{j})S(t_{j}-t_{j-1})\xi_{0}\\
    &+S(t-t_{k})\sum_{i=1}^{k}\prod_{j=k}^{i+1}(\mathcal{I}+B_{j})S(t_{j}-t_{j-1})
	(\mathcal{I}+B_{i})\int_{t_{i-1}}^{t_{i}}S(t_{i}-s)B \\
 &\times \bigg( \sum_{k=1}^{p}\Omega^{*}S^{*}(t_{k}-s)\prod_{i=k+1}^{p}S^{*}(t_{i}-t_{i-1})S^{*}(b-t_{p})\chi_{(t_{k-1},t_{k})}+\Omega^{*}S^{*}(b-s)\chi_{(t_{p},b)}\bigg)\tilde{\varphi}_{\alpha}ds\\
    &+S(t-t_{k})\sum_{i=1}^{k}\prod_{j=k}^{i+1}(\mathcal{I}+B_{j})S(t_{j}-t_{j-1})
	(\mathcal{I}+B_{i})\int_{t_{i-1}}^{t_{i}}S(t_{i}-s)\kappa(s)ds\\
  &+S(t-t_{k})\sum_{i=2}^{k}\prod_{j=k}^{i}(\mathcal{I}+B_{j}) S(t_{j}-t_{j-1}) D_{i-1}D^{*}_{i-1}\prod_{k=i}^{p}S^{*}(t_{k} -t_{k-1})(I +B^{*}_k)S^{*}(b-t_k)\tilde{\varphi}_{\alpha}\\
  &+S(t-t_{k})D_{k}D^{*}_{k}S^{*}(b-t_{k})\tilde{\varphi}_{\alpha}+\int_{t_{k}}^{t} S(t-s)\kappa(s)ds\\
  &+\int_{t_{k}}^{t} S(t-s)B\bigg( \sum_{k=1}^{p}\Omega^{*}S^{*}(t_{k}-s)\prod_{i=k+1}^{p}S^{*}(t_{i}-t_{i-1})S^{*}(b-t_{p})\chi_{(t_{k-1},t_{k})}+\Omega^{*}S^{*}(b-s)\chi_{(t_{p},b)}\bigg)\tilde{\varphi}_{\alpha} ds\\
  &=S(t-t_{k})\prod_{j=k}^{1}(\mathcal{I}+B_{j})S(t_{j}-t_{j-1})\xi_{0}+S(t-t_{k})\sum_{i=1}^{k}\prod_{j=k}^{i+1}(\mathcal{I}+B_{j})S(t_{j}-t_{j-1})
	(\mathcal{I}+B_{i})\int_{t_{i-1}}^{t_{i}}S(t_{i}-s)\kappa(s)ds\\
 &+\int_{t_{k}}^{t} S(t-s)\kappa(s)ds+\Big(\Theta^{t_{p}}_{0}+\Gamma^{b}_{t_{p}}+\tilde{\Theta}^{t_{p}}_{0}+\tilde{\Gamma}^{b}_{t_{p}}\Big)S^{*}(b-t)\Big(\alpha\mathcal{I}+\Theta^{t_{p}}_{0}+\Gamma^{b}_{t_{p}}+\tilde{\Theta}^{t_{p}}_{0}+\tilde{\Gamma}^{b}_{t_{p}}\Big)^{-1}p,
  \end{align*}
  where
\begin{align*}
p=&h-S(b-t_{p})\prod_{j=p}^{1}(\mathcal{I}+B_{j})S(t_{j}-t_{j-1})\xi_{0}\\
    -&S(b-t_{p})\sum_{i=1}^{p}\prod_{j=p}^{i+1}(\mathcal{I}+B_{j})S(t_{j}-t_{j-1})
	(\mathcal{I}+B_{i})\int_{t_{i-1}}^{t_{i}}S(t_{i}-s)\kappa(s,\xi(s))ds\\
 -&\int_{t_{p}}^{b} S(b-s)\kappa(s,\xi(s))ds.
\end{align*}
By substituting \( t = b \) into the latter equation and solving for \( z(b; \xi_0) - h \), we  obtain at equation \eqref{38}.

\bigskip

 \begin{align*}
    z(b,\xi_{0})
 &=S(b-t_{k})\prod_{j=k}^{1}(\mathcal{I}+B_{j})S(t_{j}-t_{j-1})\xi_{0}\\
 &+S(b-t_{k})\sum_{i=1}^{k}\prod_{j=k}^{i+1}(\mathcal{I}+B_{j})S(t_{j}-t_{j-1})
	(\mathcal{I}+B_{i})\int_{t_{i-1}}^{t_{i}}S(t_{i}-s)\kappa(s)ds\\
 &+\int_{t_{k}}^{t} S(b-s)\kappa(s)ds+\Big(\Theta^{t_{p}}_{0}+\Gamma^{b}_{t_{p}}+\tilde{\Theta}^{t_{p}}_{0}+\tilde{\Gamma}^{b}_{t_{p}}\Big)\Big(\alpha\mathcal{I}+\Theta^{t_{p}}_{0}+\Gamma^{b}_{t_{p}}+\tilde{\Theta}^{t_{p}}_{0}+\tilde{\Gamma}^{b}_{t_{p}}\Big)^{-1}p,
  \end{align*}
\begin{align*}
    z(b,\xi_{0}) -h
 &=S(b-t_{k})\prod_{j=k}^{1}(\mathcal{I}+B_{j})S(t_{j}-t_{j-1})\xi_{0}\\
 &+S(b-t_{k})\sum_{i=1}^{k}\prod_{j=k}^{i+1}(\mathcal{I}+B_{j})S(t_{j}-t_{j-1})
	(\mathcal{I}+B_{i})\int_{t_{i-1}}^{t_{i}}S(t_{i}-s)\kappa(s)ds\\
 &+\int_{t_{k}}^{t} S(b-s)\kappa(s)ds-h+p-\alpha\Big(\alpha\mathcal{I}+\Theta^{t_{p}}_{0}+\Gamma^{b}_{t_{p}}+\tilde{\Theta}^{t_{p}}_{0}+\tilde{\Gamma}^{b}_{t_{p}}\Big)^{-1}p\\
 &=-\alpha\Big(\alpha\mathcal{I}+\Theta^{t_{p}}_{0}+\Gamma^{b}_{t_{p}}+\tilde{\Theta}^{t_{p}}_{0}+\tilde{\Gamma}^{b}_{t_{p}}\Big)^{-1}p.
  \end{align*}
   \end{proof}
\begin{theorem}\label{u9}
Let the linear system \eqref{eq18} be $\mathcal{A}$-controllable on \( \mathscr{I} \). If the conditions \( A_0 \), \( A_1 \), and \( A_4 \) hold, then the semilinear impulsive system \eqref{eq1} is $\mathcal{A}$-controllable.
\end{theorem}
\begin{proof}
    It is clear that the conditions $A_{2}$ and $A_{3}$ can be derived from $A_{0}$. Let \(\xi^*_\alpha(\cdot)\) represent a fixed point of \(F_\alpha\) within \(B_r{(\alpha)}\). Consequently, \(\xi^*_\alpha(\cdot)\) serves as a mild solution to  \eqref{eq1} over the interval \([0, b]\), subject to the control

     \begin{align}\label{177}
    u^{*}_{\alpha}(s)&=\bigg( \sum_{k=1}^{p}\Omega^{*}S^{*}(t_{k}-s)\prod_{i=k+1}^{p}S^{*}(t_{i}-t_{i-1})S^{*}(b-t_{p})\chi_{(t_{k-1},t_{k})}+\Omega^{*}S^{*}(b-s)\chi_{(t_{p},b)}\bigg)\tilde{\varphi}^{*}_{\alpha},
\end{align}
where
\begin{align*}
    \tilde{\varphi}^{*}_{\alpha}=&\Big(\alpha\mathcal{I}+\Theta^{t_{p}}_{0}+\Gamma^{b}_{t_{p}}+\tilde{\Theta}^{t_{p}}_{0}+\tilde{\Gamma}^{b}_{t_{p}}\Big)^{-1}\bigg(h-S(b-t_{p})\prod_{j=p}^{1}(\mathcal{I}+B_{j})S(t_{j}-t_{j-1})\xi_{0}\\
    -&S(b-t_{p})\sum_{i=1}^{p}\prod_{j=p}^{i+1}(\mathcal{I}+B_{j})S(t_{j}-t_{j-1})
	(\mathcal{I}+B_{i})\int_{t_{i-1}}^{t_{i}}S(t_{i}-s)\kappa(s,\xi^{*}_{\alpha}(s))ds\\
 -&\int_{t_{p}}^{b} S(b-s)\kappa(s,\xi^{*}_{\alpha}(s))ds\bigg)
\end{align*}
and holds the following equality:
\begin{align*}
    \xi^{*}_{\alpha}(b)=h-\alpha\Big(\alpha\mathcal{I}+\Theta^{t_{p}}_{0}+\Gamma^{b}_{t_{p}}+\tilde{\Theta}^{t_{p}}_{0}+\tilde{\Gamma}^{b}_{t_{p}}\Big)^{-1}p(\xi^{*}_{\alpha}(\cdot)).
\end{align*}
To put it differently, by Lemma 3.1 $\xi_{\alpha}= \xi^{*}_{\alpha}(b)-h$  is a solution of the equation
\begin{align*}
\alpha \xi_{\alpha}+\big(\Theta^{t_{p}}_{0}+\Gamma^{b}_{t_{p}}+\tilde{\Theta}^{t_{p}}_{0}+\tilde{\Gamma}^{b}_{t_{p}}\big)\xi_{\alpha}=\alpha h_{\alpha}
\end{align*}
 with
 \begin{align*}
      h_{\alpha}=&-p(\xi^{*}_{\alpha}(\cdot))=S(b-t_{p})\prod_{j=p}^{1}(\mathcal{I}+B_{j})S(t_{j}-t_{j-1})\xi_{0}\\
    +&S(b-t_{p})\sum_{i=1}^{p}\prod_{j=p}^{i+1}(\mathcal{I}+B_{j})S(t_{j}-t_{j-1})
	(\mathcal{I}+B_{i})\int_{t_{i-1}}^{t_{i}}S(t_{i}-s)\kappa(s,\xi^{*}_{\alpha}(s))ds\\
 +&\int_{t_{p}}^{b} S(b-s)\kappa(s,\xi^{*}_{\alpha}(s))ds-h.
 \end{align*}
  By  $A_{4}$
  \begin{align*}
      \int_{0}^{b}\Vert \kappa(s,\xi^{*}_{\alpha}(s)\Vert S^{2} ds\leq N_{1}^{2}b,
  \end{align*}

As a result, the sequence \(\{\kappa(\cdot, \xi^*_{\alpha}(\cdot))\}\) is bounded and contained within \(L^2(\mathscr{I}, H)\). Therefore, there is a subsequence, which we continue to denote by \(\{\kappa(\cdot, \xi^*_{\alpha}(\cdot))\}\), that weakly converges to \(\kappa(\cdot)\) in \(L^2(\mathscr{I}, H)\). Subsequently, applying Corollary 3.3 from \cite{8}, we derive the following result:

  \begin{align*}
      \Vert h_{\alpha}-\bar{h}\Vert
 \leq &\sup_{0\leq t\leq b}\bigg\Vert S(t-t_{p})\sum_{i=1}^{p}\prod_{j=p}^{i+1}(\mathcal{I}+B_{j})S(t_{j}-t_{j-1})
	(\mathcal{I}+B_{i})\int_{t_{i-1}}^{t_{i}}S(t_{i}-s)(\kappa(s,\xi^{*}_{\alpha}(s))-\kappa(s))ds\\
 +&\int_{t_{p}}^{t} S(t-s)(\kappa(s,\xi^{*}_{\alpha}(s))-\kappa(s))ds\bigg\Vert\to 0,
  \end{align*}
  where
  \begin{align*}
      \bar{h}&=S(b-t_{p})\prod_{j=p}^{1}(\mathcal{I}+B_{j})S(t_{j}-t_{j-1})\xi_{0}\\
    +&S(b-t_{p})\sum_{i=1}^{p}\prod_{j=p}^{i+1}(\mathcal{I}+B_{j})S(t_{j}-t_{j-1})
	(\mathcal{I}+B_{i})\int_{t_{i-1}}^{t_{i}}S(t_{i}-s)\kappa(s)ds\\
 +&\int_{t_{p}}^{b} S(b-s)\kappa(s)ds-h
  \end{align*}
  Then from
  \begin{align*}
      \Vert \xi^{*}_{\alpha}(b)-h\Vert\leq& \Big\Vert \alpha\Big(\alpha\mathcal{I}+\Theta^{t_{p}}_{0}+\Gamma^{b}_{t_{p}}+\tilde{\Theta}^{t_{p}}_{0}+\tilde{\Gamma}^{b}_{t_{p}}\Big)^{-1}\bar{h}\Big\Vert\\
      +&\Big\Vert \alpha\Big(\alpha\mathcal{I}+\Theta^{t_{p}}_{0}+\Gamma^{b}_{t_{p}}+\tilde{\Theta}^{t_{p}}_{0}+\tilde{\Gamma}^{b}_{t_{p}}\Big)^{-1}\Big\Vert \Vert p(\xi^{*}_{\alpha}(\cdot))-\bar{h}\Vert\\
      \leq& \Big\Vert \alpha\Big(\alpha\mathcal{I}+\Theta^{t_{p}}_{0}+\Gamma^{b}_{t_{p}}+\tilde{\Theta}^{t_{p}}_{0}+\tilde{\Gamma}^{b}_{t_{p}}\Big)^{-1}\bar{h}\Big\Vert+ \Vert p(\xi^{*}_{\alpha}(\cdot))-\bar{h}\Vert\to 0,
  \end{align*}
  as $\alpha\to 0^{+}$. This  establishes the $\mathcal{A}$-controllability of \eqref{eq1}.
\end{proof}

\section{$\mathcal{A}$-controllability of neutral IDEs}

Impulsive neutral functional differential equations naturally extend  ordinary IDEs by incorporating both delayed effects and sudden disruptions. These equations effectively represent real-world models where the dynamics depend on historical states as well as on instantaneous disturbances. Impulsive neutral systems have seen a sharp rise in interest recently, mostly due to its useful applications in real-world industries like as chemical science, bioengineering, circuit theory, and other areas.
\bigskip

This paragraph will represent that the system \eqref{eq18} is $\mathcal{A}$-controllable if a  function \( \xi(\cdot) \in PC([-h, b], H) \) exists for any \( \alpha > 0 \) that satisfies the following requirements:

\begin{align}\label{l193}
    u_{\alpha}(s)&=\bigg( \sum_{k=1}^{p}\Omega^{*}S^{*}(t_{k}-s)\prod_{i=k+1}^{p}S^{*}(t_{i}-t_{i-1})S^{*}(b-t_{p})\chi_{(t_{k-1},t_{k})}+\Omega^{*}S^{*}(b-s)\chi_{(t_{p},b)}\bigg)\tilde{\Psi}_{\alpha},
\end{align}
\begin{align}\label{l96}
    v^{\alpha}_{p}=D^{*}_{p}S^{*}(b-t_{p})\tilde{\Psi}_{\alpha},\quad  v^{\alpha}_{k}=D^{*}_{k}\prod_{i=k}^{p}S^{*}(t_{i} -t_{i-1})(I +B^{*}_i)S^{*}(b-t_p)\tilde{\Psi}_{\alpha},\quad k=1,\dots, p-1,
\end{align}
\begin{align*}
\tilde{\Psi}_{\alpha}(\xi(\cdot))=&\Big(\alpha\mathcal{I}+\Theta^{t_{p}}_{0}+\Gamma^{b}_{t_{p}}+\tilde{\Theta}^{t_{p}}_{0}+\tilde{\Gamma}^{b}_{t_{p}}\Big)^{-1}\bigg(h-S(b-t_{p})\prod_{j=p}^{1}(\mathcal{I}+B_{j})S(t_{j}-t_{j-1})[\varphi(0)+\sigma(0,\varphi)]\\
    +&\sigma(b,\xi_{b})-S(b-t_{p})\sum_{i=1}^{p}\prod_{j=p}^{i+1}(\mathcal{I}+B_{j})S(t_{j}-t_{j-1})
	(\mathcal{I}+B_{i})\int_{t_{i-1}}^{t_{i}}S(t_{i}-s)\kappa(s,\xi(s))ds\\
 -&\int_{t_{p}}^{b} S(b-s)\kappa(s,\xi(s))ds\bigg),
\end{align*}
\begin{align}\label{eq213}
	\xi(t)=
		\begin{cases}
			S(t)[\varphi(0)+\sigma(0,\varphi)]-\sigma(b,\xi_{b})+\int_{0}^{t} S(t-s)\big[\Omega u(s)+\kappa(s,\xi(s))\big]ds,\quad  0\leq t\leq t_{1},\\
			\\
		S(t-t_{k})\xi(t^{+}_{k})+\int_{t_{k}}^{t} S(t-s)\big[\Omega u(s)+\kappa(s,\xi(s))\big]ds,\quad t_{k}<t\leq t_{k+1},\quad k=1,2,\dots, p,\\
		\end{cases}
	\end{align}
where
\begin{align}\label{eq36}
	\xi(t^{+}_{k})=&\prod_{j=k}^{1}(\mathcal{I}+B_{j})S(t_{j}-t_{j-1})[\varphi(0)+\sigma(0,\varphi)]-\sigma(b,\xi_{b})\nonumber\\
	+&\sum_{i=1}^{k}\prod_{j=k}^{i+1}(\mathcal{I}+B_{j})S(t_{j}-t_{j-1})
	(\mathcal{I}+B_{i})\int_{t_{i-1}}^{t_{i}}S(t_{i}-s)\Omega u(s)ds\nonumber\\
 +&\sum_{i=1}^{k}\prod_{j=k}^{i+1}(\mathcal{I}+B_{j})S(t_{j}-t_{j-1})
	(\mathcal{I}+B_{i})\int_{t_{i-1}}^{t_{i}}S(t_{i}-s)\kappa(s,\xi(s))ds\\
	+&\sum_{i=2}^{k}\prod_{j=k}^{i}(\mathcal{I}+B_{j}) S(t_{j}-t_{j-1}) D_{i-1}v_{i-1}+D_{k}v_{k}.\nonumber
\end{align}
\bigskip

where, we consider a sequence of intervals \(0 = t_0 < t_1 < \dots < t_p < t_{p+1} = b\), such that \(\xi(t_k)\) and \(\xi(t_k^+)\) denote the left and right limits of \(\xi(t)\) at \(t = t_k\), respectively. Let \(C = C([-h, 0], H)\) represent the set of continuous functions \(\varphi : [-\tau, 0] \to H\), equipped with the norm \(\|\varphi\| = \sup\{|\varphi(0)| : -\tau \leq 0 \leq 0\}\). For each piecewise continuous function \(y\) defined on the interval \([-\tau, b] \setminus \{t_1, \ldots, t_p\}\), and for \(t \in \mathscr{I}\), we achive \(\xi_t \in PC\) for \(t \in [0, b]\), with \(\xi_S(\theta) = \xi(t + \theta)\) for \(\theta \in [-\tau, 0]\). For any \(k = 0, 1, \dots, p\), let \(\mathscr{I}_k = [t_k, t_{k+1}]\). The space \(PC(\mathscr{I}_k, H)\) consists of all continuous functions from \(\mathscr{I}_k\) to \(H\), with the norm \(\|\xi\|_{\mathscr{I}_k} = \sup\{|\xi(t)| : t \in \mathscr{I}_k\}\).
\bigskip

Next, define the space \(PC([-\tau, b], H) = \{\xi : [-\tau, b] \to H : \xi_k \in PC(\mathscr{I}_k, H), \, k = 0, \dots, p \) and there are  $\xi(t^+)$  and $\xi(t_k^-)$ such that \(\xi(t_k) = \xi(t_k^+), k = 1, \dots, p\}\), which is a Banach space with the norm \(\|\xi\|_{PC([-\tau, b], H)} = \sup\{\|\xi_k\|_{\mathscr{I}_k} : k = 0, \dots, p\}\), where \(\xi_k\) is the restriction of \(y\) to \(\mathscr{I}_k\) for \(k = 0, \dots, p\).
\bigskip

Concering \( \kappa \) and \( \sigma \), we assume the following hypotheses:
\bigskip

$(i)$ 	There are  $\lambda_{m}(\cdot)\in L^{1}(\mathscr{I},R^{+})$ and $\psi_{m}(\cdot)\in L^{1}(PC,R^{+}), m=1,\dots, q$, for which
	\bigskip

 \begin{align*}
     \Vert \kappa(t,\varphi)\Vert\leq \sum_{m=1}^{q}\lambda_{m}(t)\psi_{m}(\xi)\quad \forall(t,\varphi)\in \mathscr{I}\times H.
 \end{align*}

\bigskip

$(ii)$ For each \(\alpha> 0\),
\[
\limsup_{r \to \infty} \left(r - \sum_{m=1}^{q} \frac{\bar{c}_{m}}{\alpha}\sup \left\{ \psi_{m}(\varphi) : \|\varphi\| \leq r \right\} \right)=\infty.
\]

\bigskip

$(iii)$ The function \( \kappa : \mathscr{I} \times C \to H \)  uniformly bounded  and continuous, meaning that there is a \( N_2 > 0 \) for which

\[
\| \kappa(t,\varphi) \| \leq N_2 \quad \forall \ (t,\varphi) \in \mathscr{I} \times C.
\]

\bigskip
$(iv)$ The function \( \sigma : \mathscr{I} \times C \to H \) is uniformly bounded  and continuous, meaning that there is a \( N_3 > 0 \) for which

\[
\| \sigma(t,\varphi) \| \leq N_3 \quad \forall \ (t,\varphi) \in \mathscr{I} \times C.
\]
\bigskip

\begin{theorem}\label{k3}

Suppose that the linear IDE \eqref{eq18} is $\mathcal{A}$-controllable on the interval \([0, b]\). If the semigroup \( S(t) \) is compact and the conditions $(i)-(iv)$  are held, then the system given by equation \eqref{eq188} will also be $\mathcal{A}$-controllable.
\end{theorem}

\begin{proof}
For \(\alpha > 0\), we assign the operator \(F_\alpha\) on \(PC([-\tau,b], H)\) as
\[
F_\alpha(\xi) = z,
\]
such that
\begin{align}
	z(t)=
		\begin{cases}
			S(t)[\varphi(0)+\sigma(0,\varphi)]-\sigma(b,\xi_{b})+\int_{0}^{t} S(t-s)\big[\Omega u_{\alpha}(s)+\kappa(s,\xi(s))\big]ds,\quad  0\leq t\leq t_{1},\\
			\\
		S(t-t_{k})\xi(t^{+}_{k})+\int_{t_{k}}^{t} S(t-s)\big[\Omega u_{\alpha}(s)+\kappa(s,\xi(s))\big]ds,\quad t_{k}<t\leq t_{k+1},\quad k=1,2,\dots, p,\\
		\end{cases}
	\end{align}
 \begin{align}
     v_{0}(\theta)=\varphi(\theta),\quad \theta\in [-\tau,0],
 \end{align}
where
\begin{align}
	\xi(t^{+}_{k})=&\prod_{j=k}^{1}(\mathcal{I}+B_{j})S(t_{j}-t_{j-1})[\varphi(0)+\sigma(0,\varphi)]-\sigma(b,\xi_{b})\nonumber\\
	+&\sum_{i=1}^{k}\prod_{j=k}^{i+1}(\mathcal{I}+B_{j})S(t_{j}-t_{j-1})
	(\mathcal{I}+B_{i})\int_{t_{i-1}}^{t_{i}}S(t_{i}-s)\Omega u_{\alpha}(s)ds\nonumber\\
 +&\sum_{i=1}^{k}\prod_{j=k}^{i+1}(\mathcal{I}+B_{j})S(t_{j}-t_{j-1})
	(\mathcal{I}+B_{i})\int_{t_{i-1}}^{t_{i}}S(t_{i}-s)\kappa(s,\xi(s))ds\\
	+&\sum_{i=2}^{k}\prod_{j=k}^{i}(\mathcal{I}+B_{j}) S(t_{j}-t_{j-1}) D_{i-1}v_{i-1}+D_{k}v_{k},\nonumber
\end{align}
\begin{align}
    u_{\alpha}(s)&=\bigg( \sum_{k=1}^{p}\Omega^{*}S^{*}(t_{k}-s)\prod_{i=k+1}^{p}S^{*}(t_{i}-t_{i-1})S^{*}(b-t_{p})\chi_{(t_{k-1},t_{k})}+\Omega^{*}S^{*}(b-s)\chi_{(t_{p},b)}\bigg)\tilde{\Psi}_{\alpha},
\end{align}
\begin{align*}
\tilde{\Psi}_{\alpha}(\xi(\cdot))=&\Big(\alpha\mathcal{I}+\Theta^{t_{p}}_{0}+\Gamma^{b}_{t_{p}}+\tilde{\Theta}^{t_{p}}_{0}+\tilde{\Gamma}^{b}_{t_{p}}\Big)^{-1}\bigg(h-S(b-t_{p})\prod_{j=p}^{1}(\mathcal{I}+B_{j})S(t_{j}-t_{j-1})[\varphi(0)+\sigma(0,\varphi)]\\
    +&\sigma(b,\xi_{b})-S(b-t_{p})\sum_{i=1}^{p}\prod_{j=p}^{i+1}(\mathcal{I}+B_{j})S(t_{j}-t_{j-1})
	(\mathcal{I}+B_{i})\int_{t_{i-1}}^{t_{i}}S(t_{i}-s)\kappa(s,\xi(s))ds\\
 -&\int_{t_{p}}^{b} S(b-s)\kappa(s,\xi(s))ds\bigg).
\end{align*}
\bigskip

    It is not difficult to indicate that if  \(F\) admits a fixed point for all \(\alpha > 0\) using the procedure from the prior section, then one can conclude that system \eqref{eq18} is $\mathcal{A}$-controllable by applying the approach found in Theorem \ref{u9}.
\end{proof}

Future research on the $\mathcal{A}$-controllability of neutral IDEs could focus on the following areas:

\textbf{Fractional Neutral Systems:} Extending controllability results to impulsive neutral systems with fractional derivatives to address processes with memory and hereditary properties.

\textbf{Variable-Order Dynamics:} Exploring impulsive neutral systems with variable-order derivatives to model complex, time-dependent dynamics more accurately.

\textbf{Stochastic Influences:} Investigating the controllability of impulsive neutral systems under stochastically perturbed  uncertainties.

\textbf{Nonlinear and Multi-Valued Maps:} Studying systems with nonlinear or multi-valued operators to address challenges in fields like material science, population dynamics, and control engineering.

\textbf{Infinite-Dimensional Systems:} Analyzing impulsive neutral systems in infinite-dimensional spaces, such as those governed by PDEs or delay differential equations.

\textbf{Optimal Control Strategies:} Combining controllability analysis with optimization techniques to design efficient control strategies for resource-constrained systems.

\textbf{Applications with Non-Compactness Measures:} Focusing on systems where the measure of non-compactness is critical, providing deeper insights into approximate controllability in more complex settings.

\textbf{Hybrid and Switched Neutral Systems:} Examining the controllability of hybrid and switched impulsive neutral systems to reflect diverse operational modes and transitions.

\section{Applications}
\begin{theorem}
If  $b-t_{p}\geq 2\pi$, $\gamma_{m} \not=0 $ for $m = 1,2, \dots$ under the assumptions $A_{0}, A_{1}$ and $A_{4}$, then system \eqref{u10}
\begin{align}\label{u10}
    \begin{cases}
        \frac{\partial^{2}\eta(t,\theta)}{\partial S^{2}}=  \frac{\partial^{2}\eta(t,\theta)}{\partial\theta^{2}}+hu(t)+\kappa(t,\eta(t,\theta)),\\
        \eta(t,0)=\eta(t,\pi) \equiv  0,\\
        \eta(0,\theta)=a(\theta),\quad \frac{\partial \eta(0,\theta)}{\partial t}=b(\theta),\\
        \Delta \eta(t_{i},\theta)=a_{i}(\theta),\quad  \Delta \frac{\partial \eta(t_{i},\theta)}{\partial t}=b_{i}(\theta),\quad i=1,\dots, p.
    \end{cases}
\end{align}
 is $\mathcal{A}$-controllable on $\mathscr{I}$.
\end{theorem}
\begin{proof}

To analyze the system, we start by expanding the initial conditions \(a(\theta)\) and \(b(\theta)\) in terms of a Fourier series:

\[
a(\theta) = \sum_{m=1}^{\infty} \alpha_m \sin(m \theta), \quad b(\theta) = \sum_{m=1}^{\infty} \beta_m \sin(m \theta), \quad \theta \in (0, \pi).
\]

Where, \(\alpha_m\) and \(\beta_m\) are the Fourier coefficients that capture the spatial dependence of the initial data.

 \bigskip

For the corresponding linear system, we can express \(\eta(t, \theta)\) as a series involving trigonometric functions of time \(t\) and spatial functions \(\sin(m \theta)\):

\[
\eta(t, \theta) = \sum_{m=1}^{\infty} \left(  \alpha_m \cos(m t) + \frac{\beta_m}{m} \sin(m t) \right) \sin(m \theta).
\]

The time derivative of \(\eta(t, \theta)\) is given by:
\[
\frac{\partial \eta(t, \theta)}{\partial t} = \sum_{m=1}^{\infty} \left(  -m \alpha_m \sin(m t) + \beta_m \cos(m t) \right) \sin(m \theta).
\]

\bigskip

Define the Hilbert space \(H\) of initial conditions as the set of pairs \(\begin{pmatrix} a \\ b \end{pmatrix}\) of functions with expansions \(a(\theta)\) and \(b(\theta)\) such that

\[
\sum_{m=1}^{\infty} \left(  m^2 |\alpha_m|^2 + |\beta_m|^2 \right) < \infty.
\]

This space \(H\) is equipped with the dot product
\[
\left\langle \begin{pmatrix} a \\ b \end{pmatrix}, \begin{pmatrix} \tilde{a} \\ \tilde{b} \end{pmatrix} \right\rangle = \sum_{m=1}^{\infty} \left( S( m^2 \alpha_m \tilde{\alpha}_m + \beta_m \tilde{\beta}_m \right).
\]

For the linearized system, the semigroup of solutions \(S(t)\) can be defined as:
\[
S(t) \begin{pmatrix} a \\ b \end{pmatrix} = \sum_{m=1}^{\infty} \begin{pmatrix} \cos(m t) & \frac{1}{m} \sin(m t) \\ -m \sin(m t) & \cos(m t) \end{pmatrix} \begin{pmatrix} \alpha_m \\ \beta_m \end{pmatrix} \sin(m \theta), \quad t \geq 0.
\]

This semigroup represents the evolution of initial states under the linear part of the system and is significative $\forall$ \(t \in \mathbb{R}\) and $S^{*}(t)=S^{-1}(t)=S(-t),\quad t\in R$.

\bigskip

Using Duhamel's principle, the mild solution of the nonlinear system can be written as:

\[
\begin{pmatrix} \eta_1(t) \\ \eta_2(t) \end{pmatrix} = S(t) \begin{pmatrix} a \\ b \end{pmatrix} + \int_0^t S(t - s) \begin{pmatrix} 0 \\ h \end{pmatrix} u(s) \, ds + \int_0^t S(t - s) \begin{pmatrix} 0 \\ \kappa(s, \eta(s, \theta)) \end{pmatrix} ds.
\]

In this context, we define the control space as \( U = \mathbb{R} \), with the operator \( \Omega : \mathbb{R} \to H \) specified by \( \Omega u = \begin{pmatrix} 0 \\ h \end{pmatrix} u \) for \( u \in \mathbb{R} \). Moreover, the semigroup satisfies the property \( S^*(t) = S(-t) \) for all \( t \geq 0 \).

Given the expression

\[
\Omega^*S^*(b - t) \begin{pmatrix} a \\ b \end{pmatrix} = \sum_{m=1}^{\infty} \gamma_m \left( m \alpha_m \sin(m(b - t)) + \beta_m \cos(m(b - t)) \right), \quad t_p \leq t \leq b,
\]

where, we define the right-hand series as \( \varphi(t) \) for \( 0 \leq t \leq b - t_p \), representing a continuous and periodic function with period \( 2\pi \). Furthermore, the coefficients satisfy

\[
m \gamma_m \alpha_m = \frac{1}{\pi} \int_0^{2\pi} \varphi(t) \cos(mt) \, dt, \quad \gamma_m \beta_m = \frac{1}{\pi} \int_0^{2\pi} \varphi(t) \sin(mt) \, dt, \quad m = 1, 2, \dots
\]

Assuming \( b \geq t_p + 2\pi \) and \( \varphi(t) = 0 \) for \( 0 \leq t \leq b - t_p \), we obtain \( m \gamma_m \alpha_m = 0 \) and \( \gamma_m \beta_m = 0 \) for \( m = 1, 2, \dots \). Since \( \gamma_m \neq 0 \), it follows that \( \alpha_m = \beta_m = 0 \) for all \( m \), leading to the conclusion \( a = b = 0 \).
\bigskip

Thus, by Theorem \ref{u9}, we conclude that the wave equation \eqref{u10} is $\mathcal{A}$-controllable.
\end{proof}
\bigskip

\begin{example}
 Analyze a control system that is subject to impulsive effects and is regulated by the heat equation:
 \begin{align}\label{k1}
     \begin{cases}
         \frac{\partial \eta(t,z)}{\partial t}= \frac{\partial^{2} \eta(t,z)}{\partial z^{2}}+\Omega u(t,z)+\kappa(t,\eta(t,z)),\quad 0<z<\pi,\\
         \eta(t,0)=\eta(t,\pi)=0,\quad t\in [0,b]\setminus \{t_{1},\dots, t_{p}\},\\
         \eta(0,z)=\eta_{0}(z),\quad z\in [0,\pi],\\
         \Delta \eta(t_{k},z)=-\eta(t_{k},z)-v_{k}(z), \quad z\in (0,\pi),\quad k=1,\dots,p-1.
     \end{cases}
 \end{align}

 Let \( H = L^2[0, \pi] \), and consider the operator \( A: H \rightarrow H \) defined by \( A\eta = \eta^{\prime\prime} \). The domain of \( A \) is given by

\[
D(A) = \{ w \in H : w \text{ and } w' \text{ are absolutely continuous, } w'' \in H, \text{ and } w(0) = w(\pi) = 0 \}.
\]

Let \( A \) be an operator on \( H = L^2[0, \pi] \) defined by

\[
Aw = -\sum_{n=1}^{\infty} n^{2} \langle w, e_n \rangle e_n, \quad w \in D(A),
\]
where \( \lambda_n = n^2 \), and \( e_n(z) = \sqrt{\frac{2}{\pi}} \sin(nz) \) for \( 0 \leq z \leq \pi \), \( n = 1, 2, \dots \).
\bigskip

Given an initial state \( v_0 \in L^2[0, \pi] \), let the impulsive term  defined as \(\Delta \eta(t_{k},z)=  -\eta(t_k, z)-v_{k}(z) \), such that $B_{k}=D_{k}=-I$ and let \( \kappa \) be a function that is Lipschitz continuous and satisfies linear growth conditions. It is a recognized fact that \( A \) generates a compact semigroup \( S(t) \) in \( H \), represented as

\[
S(t)w = \sum_{n=1}^{\infty} e^{-n^{2}} \langle w, e_n \rangle e_n,\quad w\in H.
\]

Define an infinite-dimensional space \( U \) by

\[
U = \left\{ u  : u = \sum_{n=2}^{\infty} u_n e_n, \quad \sum_{n=2}^{\infty} u_n^2 < \infty \right\},
\]

with the norm \( \|u\|_U = \left(  \sum_{n=2}^{\infty} u_n^2 \right)^{1/2} \). Then define a mapping \( \Omega: U \to H \) by

\[
\Omega u = 2u_2 e_1 + \sum_{n=2}^{\infty} u_n e_n.
\]

Due to the compactness of the semigroup \( S(t) \) generated by \( A \), the associated linear system lacks exact controllability but achieves $\mathcal{A}$-controllability, as noted in \cite{1}. This implies that we can express system \eqref{k1} in the abstract form of equation \eqref{eq1}. According to Theorem \ref{u9}, this system is therefore $\mathcal{A}$-controllable over the interval \([0, b]\).
\end{example}
\bigskip

\begin{example}
    We examine the controlled neutral differential equation under the impulsive effects that follows:

\begin{align}\label{k2}
     \begin{cases}
         \frac{\partial }{\partial t}\big[\eta(t,z)-\lambda_{1}(t,\eta(t-\tau,z))\big]= \frac{\partial^{2} }{\partial z^{2}} \big[\eta(t,z)-\lambda_{1}(t,\eta(t-\tau,z))\big] \\
         \quad  \quad \quad \quad \quad \quad \quad \quad \quad \quad \quad \quad \quad+\nu(t,z)+\lambda_{2}(t,\eta(t-\tau,z)),& 0<z<1,\\
         \eta(t,0)=\eta(t,1)=0,& t>1,\\
         \eta(t,z)=\varphi(t,z),& t\in [-\tau,0],\\
         \Delta \eta(t_{k},z)=-\eta(t_{k},z)-v_{k}(z), &  k=1,\dots,p-1.
     \end{cases}
 \end{align}
Let \( \sigma(t, w(t))(z) = \lambda_1(t, w(t -z)) \) and \( \kappa(t, w(t))(z) = \lambda_2(t, w(t -z)) \). Define the operator \((\Omega u)(t)(z) = \nu(t, z) \, \), where \( z \in (0, 1) \).
\bigskip

Consider \( H = L^2[0, 1] \) and set the operator \( A: H \to H \) by the differential equation:
\[
\frac{d^2 w}{dz^2} = Aw
\]
with the domain
\[
D(A) = \Big\{ w \in H \mid w \text{ is absolutely continuous, } \frac{d^2w}{dz^2} \in H, \, \frac{d w}{dt}(0) = \frac{d w}{dt}(1) = 0 \Big\}.
\]
The operator \( A \) has eigenvalues given by \( \eta_n = -n^2\pi^2 \) for \( n \geq 0 \) and corresponding eigenvectors \( e_n(z) = \sqrt{2} \cos(n\pi z) \) for \( n \geq 1 \), with \( e_0 = 1 \), forming an orthonormal basis for \( L^2(0, 1) \). It is well known that \( A \) generates a compact semigroup \( S(t) \) in \( H \), defined by:
\[
S(t)w = \int_0^1 w(v) \, dv + \sum_{n=1}^\infty e^{-n^2 \pi^2 t} \cos(\pi n z) \int_0^1 \cos(\pi n z) w(v) \, dv, \quad w \in H.
\]

The functions \( \lambda_1, \lambda_2: [0, 1]\times [0,1] \to [0, 1] \) are continuous, and there are constants \( k_1 \) and \( k_2 \) such that:
\[
\|\lambda_1(t, w(t -z))\| \leq k_1 \quad \text{and} \quad \|\lambda_2(t, w(t -z))\| \leq k_2.
\]

Therefore, Equation \eqref{k2} can be rewritten in the form of \eqref{eq188} using the previously defined operator \( A \), functions \( \sigma \), and function \( \kappa \). The linear system associated with Equation \eqref{k2} exhibits (($\mathcal{A}$-controllability. According to Theorem \ref{k3}, we accomplish that the system represented by  \eqref{k2} is indeed $\mathcal{A}$-controllable.
\end{example}
\bigskip
In the following example, we present  specific case that illustrates the solution of equation \eqref{eq1} within finite-dimensional Hilbert spaces. This approach clarifies the influence of impulses on the solution of the equation. By comparing the non-impulsive case with the impulsive cases, it becomes evident how each impulse and control input affects the system’s dynamics over time.
\bigskip

\begin{example}
    It is obvious that, in the finite-dimensional Hilbet space the solution of equation \eqref{eq1} is given by:

	\begin{align}\label{eq22}
	\xi(t)=
		\begin{cases}
			e^{At}\xi(0)+\int_{0}^{t} e^{A(t-s)}\big[\Omega u(s)+\kappa(s,\xi(s))\big]ds,\quad  0\leq t\leq t_{1},\\
			\\
		e^{A(t-t_{k})}\xi(t^{+}_{k})+\int_{t_{k}}^{t} e^{A(t-s)}\big[\Omega u(s)+\kappa(s,\xi(s))\big]ds,\quad t_{k}<t\leq t_{k+1},\quad k=1,2,\dots, p,\\
		\end{cases}
	\end{align}
where
\begin{align}\label{eq33}
	\xi(t^{+}_{k})=&\prod_{j=k}^{1}(\mathcal{I}+B_{j})e^{A(t_{j}-t_{j-1})}\xi_{0}\nonumber\\
	+&\sum_{i=1}^{k}\prod_{j=k}^{i+1}(\mathcal{I}+B_{j})e^{A(t_{j}-t_{j-1})}
	(\mathcal{I}+B_{i})\int_{t_{i-1}}^{t_{i}}e^{A(t_{i}-s)}\Omega u(s)ds\nonumber\\
 +&\sum_{i=1}^{k}\prod_{j=k}^{i+1}(\mathcal{I}+B_{j})e^{A(t_{j}-t_{j-1})}
	(\mathcal{I}+B_{i})\int_{t_{i-1}}^{t_{i}}e^{A(t_{i}-s)}\kappa(s,\xi(s))ds\\
	+&\sum_{i=2}^{k}\prod_{j=k}^{i}(\mathcal{I}+B_{j}) e^{A(t_{j}-t_{j-1})} D_{i-1}v_{i-1}+D_{k}v_{k}.\nonumber
\end{align}

Let \( H = \mathbb{R}^2 \). The initial condition is given by \( \xi(0) = \begin{pmatrix} 1 \\ 0 \end{pmatrix} \).
Define the operator \( A \) generating the semigroup \( S(t) \) as:

\[
A = \begin{pmatrix}
0 & 1 \\
-1 & 0
\end{pmatrix}
\]

This operator represents a rotation and defines the semigroup:

\[
S(t) = e^{At} = \begin{pmatrix}
\cos(t) & \sin(t) \\
-\sin(t) & \cos(t)
\end{pmatrix}
\]

Let the control function \( u(t) \) and nonlinear function \( \kappa(t, \xi(t)) \)   be defined as:

\[
u(t) = \begin{pmatrix}
1 \\
0
\end{pmatrix},
\quad
\kappa(t, \xi(t)) = \begin{pmatrix}
0 \\
0.1 \xi_1^2(t)
\end{pmatrix}\quad \text{for } t \in [0, 2],
\]

where \( \xi = \begin{pmatrix} \xi_1 \\ \xi_2 \end{pmatrix} \).
\bigskip

Assume there is a single impulsive point at \( t_1 = 1 \). And we  define the impulsive operator \( B_1 \), \( D_1 \) and the function $v_{1}$ as:

\[
B_1 = \begin{pmatrix}
0 & 0 \\
0 & -0.5
\end{pmatrix}, \quad D_1 = \begin{pmatrix}
1 \\
0
\end{pmatrix},\quad v_{1}=1.
\]

1. For \( 0 \leq t < t_1 \):
   We compute \( \xi(t) \) using equation \(\eqref{eq22}\):

   \[
   \xi(t) = S(t)\xi(0) + \int_{0}^{t} S(t-s) \left[ \Omega u(s) + \kappa(s,\xi(s)) \right] ds
   \]

   The \( B \) operator can be taken as:

   \[
   B = \begin{pmatrix}
   1 & 0 \\
   0 & 0
   \end{pmatrix}
   \]

   Thus, the solution in this interval becomes:

   \begin{align*}
   \xi(t) &= \begin{pmatrix}
   \cos(t) \\
   -\sin(t)
   \end{pmatrix} + \int_{0}^{t} \begin{pmatrix}
   \cos(t-s) & \sin(t-s) \\
   -\sin(t-s) & \cos(t-s)
   \end{pmatrix}  \begin{pmatrix} 1 \\ 0.1 \xi^{2}_{1}(s) \end{pmatrix}  ds
    \end{align*}

2. For \( t = 1 \):
   We compute \( \xi(t^+_1) \) using equation \(\eqref{eq33}\):

   \begin{align*}
   \xi(t^+_1) &= (\mathcal{I} + B_1) S(1)\xi(0) + (\mathcal{I} + B_1) \int_{0}^{1} S(1-s) [\Omega u(s)+ \kappa(s,\xi(s))] ds+D_{1}v_{1}\\
   &=\begin{pmatrix}
1 & 0 \\
0 & 0.5
\end{pmatrix}\begin{pmatrix}
   \cos(1) \\
   -\sin(1)
   \end{pmatrix}+\begin{pmatrix}
1 & 0 \\
0 & 0.5
\end{pmatrix}\int_{0}^{1} \begin{pmatrix}
   \cos(1-s) & \sin(1-s) \\
   -\sin(1-s) & \cos(1-s)
   \end{pmatrix}  \begin{pmatrix} 1 \\ 0.1 \xi^{2}_{1}(s) \end{pmatrix}  ds+\begin{pmatrix}
1 \\
0
\end{pmatrix}\\
&\approx \begin{pmatrix}
1,5403 \\
-0.42075
\end{pmatrix}+\begin{pmatrix}
1 & 0 \\
0 & 0.5
\end{pmatrix}\int_{0}^{1} \begin{pmatrix}
   \cos(1-s) & \sin(1-s) \\
   -\sin(1-s) & \cos(1-s)
   \end{pmatrix}  \begin{pmatrix} 1 \\ 0.1 \xi^{2}_{1}(s) \end{pmatrix}  ds.
    \end{align*}

3. For \( 1 < t \leq 2 \):
   Now we compute \( \xi(t) \):

   \begin{align*}
   \xi(t) &= S(t - 1)\xi(t^+_1) + \int_{1}^{t} S(t - s) \left[ \Omega u(s) + \kappa(s,\xi(s)) \right] ds\\
   &\approx \begin{pmatrix}
   \cos(t-1) & \sin(t-1) \\
   -\sin(t-1) & \cos(t-1)
   \end{pmatrix} \begin{pmatrix}
1,5403 \\
-0.42075
\end{pmatrix}\\
&+\begin{pmatrix}
   \cos(t-1) & \sin(t-1) \\
   -\sin(t-1) & \cos(t-1)
   \end{pmatrix}\begin{pmatrix}
1 & 0 \\
0 & 0.5
\end{pmatrix}\int_{0}^{1} \begin{pmatrix}
   \cos(1-s) & \sin(1-s) \\
   -\sin(1-s) & \cos(1-s)
   \end{pmatrix}  \begin{pmatrix} 1 \\ 0.1 \xi^{2}_{1}(s) \end{pmatrix}  ds\\
   &+ \int_{1}^{t} \begin{pmatrix}
   \cos(t-s) & \sin(t-s) \\
   -\sin(t-s) & \cos(t-s)
   \end{pmatrix}  \begin{pmatrix} 1 \\ 0.1 \xi^{2}_{1}(s) \end{pmatrix} ds.
    \end{align*}

   The final form of \( \xi(t) \) will depend on the computations made in the integral from \( 1 \) to \( t \).

This example demonstrates a finite-dimensional impulsive system with a single impulse at \( t_1 = 1 \). The solution illustrates how the system evolves continuously until the impulse occurs and then adjusts the state variable accordingly.
And their graph describe in Figure 1 and 2 with impulsive and non impulsive cases.

\begin{figure}[h!t]
    \centering
    \includegraphics[width=16cm]{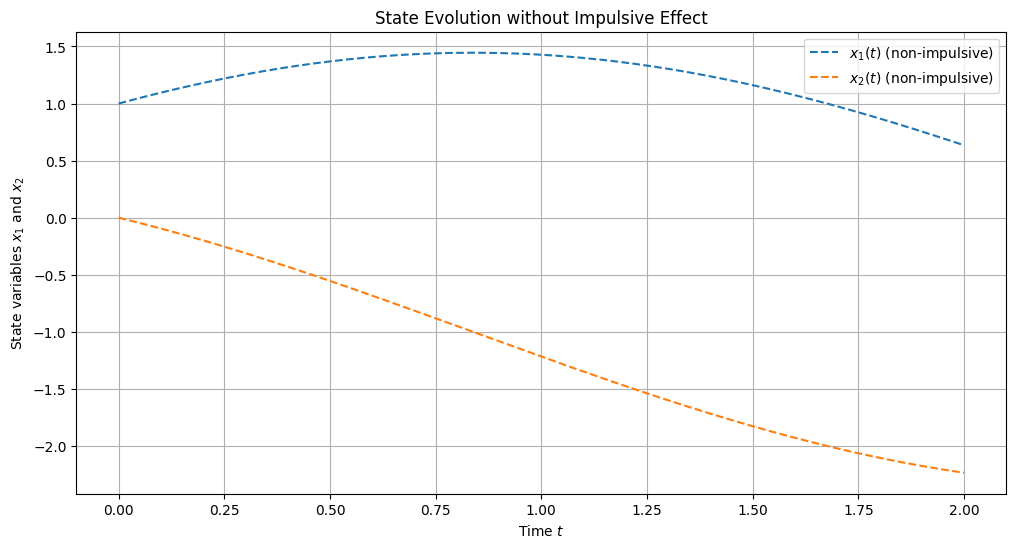}  
    \caption{}
\end{figure}

\begin{figure}[h!t]
    \centering
    \includegraphics[width=16cm]{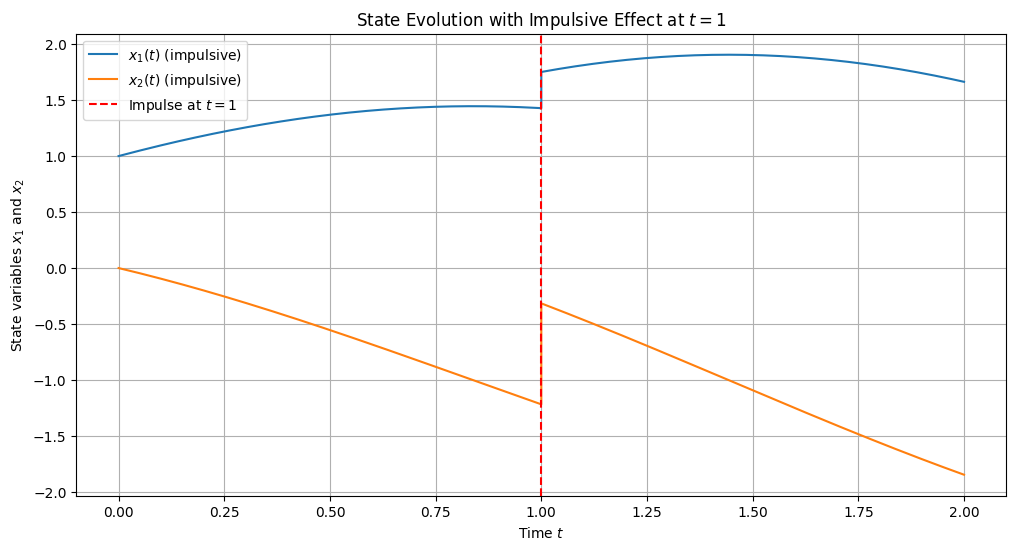}  
    \caption{}
\end{figure}

 4. For \( u(t) = 0 \):
\bigskip

If \( u(t) \) is set to zero, the control function changes to:

\[
u(t) = \begin{pmatrix}
0 \\
0
\end{pmatrix}
\]

In this scenario, the equations for the impulsive system need to be adjusted. Specifically, the equation for \( \xi(t) \) for \( 0 \leq t < t_1 \) becomes:

\begin{align*}
    \xi(t)& = S(t)\xi(0) + \int_{0}^{t} S(t-s) \kappa(s,\xi(s)) ds\\
&= \begin{pmatrix}
   \cos(t) \\
   -\sin(t)
   \end{pmatrix} + \int_{0}^{t} \begin{pmatrix}
   \cos(t-s) & \sin(t-s) \\
   -\sin(t-s) & \cos(t-s)
   \end{pmatrix}  \begin{pmatrix} 0 \\ 0.1 \xi^{2}_{1}(s) \end{pmatrix}  ds
\end{align*}

This implies that the state evolution is governed solely by the semigroup dynamics and the nonlinear function \( \kappa(s, \xi(s)) \), without any external control input.

For the impulse effect at \( t = 1 \), we compute \( \xi(t^+_1) \) as follows:

\begin{align*}
\xi(t^+_1) &= (\mathcal{I} + B_1) S(1)\xi(0) + \int_{0}^{1} S(1-s) \kappa(s,\xi(s)) ds+D_{1}v_{1}\\
 &=\begin{pmatrix}
1 & 0 \\
0 & 0.5
\end{pmatrix}\begin{pmatrix}
   \cos(1) \\
   -\sin(1)
   \end{pmatrix}+\begin{pmatrix}
1 & 0 \\
0 & 0.5
\end{pmatrix}\int_{0}^{1} \begin{pmatrix}
   \cos(1-s) & \sin(1-s) \\
   -\sin(1-s) & \cos(1-s)
   \end{pmatrix}  \begin{pmatrix} 0 \\ 0.1 \xi^{2}_{1}(s) \end{pmatrix}  ds+\begin{pmatrix}
1 \\
0
\end{pmatrix}\\
&\approx \begin{pmatrix}
1,5403 \\
-0.42075
\end{pmatrix}+\begin{pmatrix}
1 & 0 \\
0 & 0.5
\end{pmatrix}\int_{0}^{1} \begin{pmatrix}
   \cos(1-s) & \sin(1-s) \\
   -\sin(1-s) & \cos(1-s)
   \end{pmatrix}  \begin{pmatrix} 0 \\ 0.1 \xi^{2}_{1}(s) \end{pmatrix}  ds.
\end{align*}

For the interval \( 1 < t \leq 2 \), the state evolves according to:

\begin{align*}
\xi(t) &= S(t - 1)\xi(t^+_1) + \int_{1}^{t} S(t - s) \kappa(s,\xi(s)) ds\\
  &\approx \begin{pmatrix}
   \cos(t-1) & \sin(t-1) \\
   -\sin(t-1) & \cos(t-1)
   \end{pmatrix} \begin{pmatrix}
1,5403 \\
-0.42075
\end{pmatrix}\\
&+\begin{pmatrix}
   \cos(t-1) & \sin(t-1) \\
   -\sin(t-1) & \cos(t-1)
   \end{pmatrix}\begin{pmatrix}
1 & 0 \\
0 & 0.5
\end{pmatrix}\int_{0}^{1} \begin{pmatrix}
   \cos(1-s) & \sin(1-s) \\
   -\sin(1-s) & \cos(1-s)
   \end{pmatrix}  \begin{pmatrix} 0 \\ 0.1 \xi^{2}_{1}(s) \end{pmatrix}  ds\\
   &+ \int_{1}^{t} \begin{pmatrix}
   \cos(t-s) & \sin(t-s) \\
   -\sin(t-s) & \cos(t-s)
   \end{pmatrix}  \begin{pmatrix} 0 \\ 0.1 \xi^{2}_{1}(s) \end{pmatrix} ds.
\end{align*}

This example demonstrates a finite-dimensional impulsive system with a single impulse at \( t_1 = 1 \). The solution illustrates how the system evolves continuously until the impulse occurs and then adjusts the state variable accordingly.
And their graph describe in Figure 3 and 4 ($u=0$)   with impulsive and non impulsive cases.

\begin{figure}[h!t]
    \centering
    \includegraphics[width=16cm]{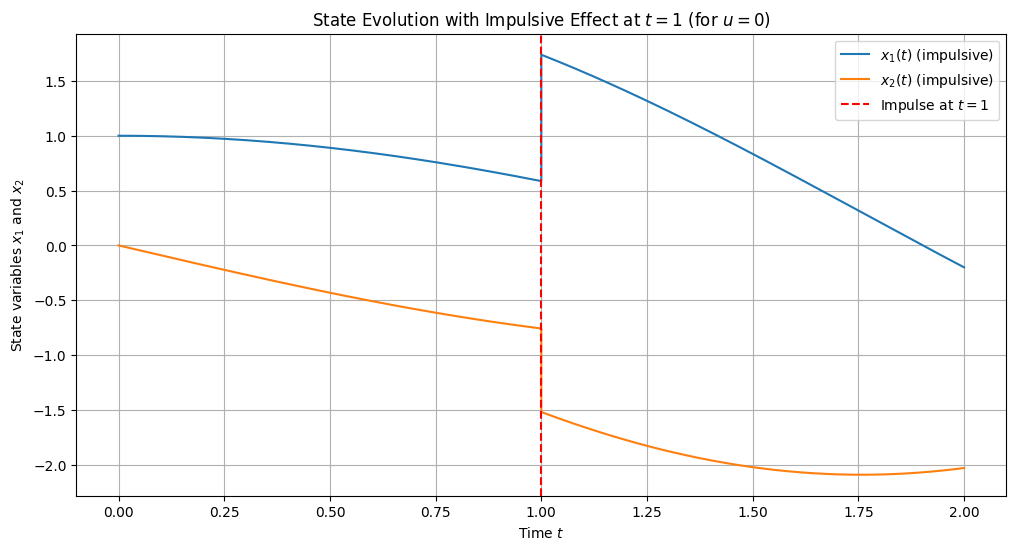}  
    \caption{}
\end{figure}
\begin{figure}[h!t]
    \centering
    \includegraphics[width=16cm]{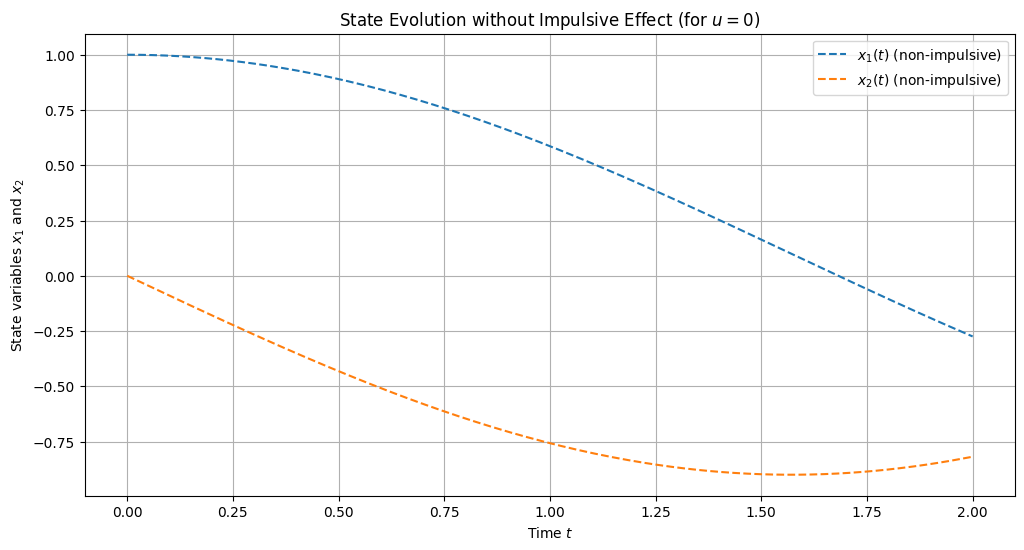}  
    \caption{}
\end{figure}
\end{example}

\section{Conclusion}

The $\mathcal{A}$-controllability of some neutral and semi-linear differential equations with control under impulsive effects was examined in Hilbert spaces in this study. These impulsive semi-linear and neutral differential equations were found to have sufficient requirements for $\mathcal{A}$-controllability using semigroup theory and a fixed-point method. Three examples were given to illustrate how the findings can be used in practice, showing improvements over some recent findings.

$\mathcal{A}$-controllability of impulsive systems refers to the ability to steer the system's state close to any desired target state, even if it cannot reach the target exactly, in systems that experience sudden changes (impulses) at specific times. These impulses represent abrupt events—such as shocks or jumps—that cause an immediate alteration in the system’s state.

For impulsive systems, $\mathcal{A}$-controllability requires analyzing both the continuous dynamics of the system and the effects of impulses. To establish conditions for $\mathcal{A}$-controllability, techniques like fixed-point theorems, semigroup theory, and resolvent operators are often employed. These methods help characterize whether the system’s state can be controlled within a desired proximity to the target state despite the discontinuous behavior caused by impulses.


Challenges in solid mechanics, frequently involve non-monotonic and multi-valued constitutive laws, leading to fractional inclusions. The findings discussed here can be addressed to investigate the approximate and finite approximate controllability of neutral IDES and inclusions by appropriately defining a multi-valued map.

For future research directions, we plan to integrate the above analysis with topics such as fractional differential inclusions, fractional discrete calculus, and variable-order derivatives.

Future research on the  controllability of IDEs systems could focus on the following directions:

\textbf{Variable-Order Systems:} Investigating impulsive systems with variable-order derivatives to capture dynamic processes with varying memory effects.

\textbf{Fractional Dynamics:} Extending controllability results to fractional impulsive systems, including those with distributed delays or complex boundary conditions.

\textbf{Hybrid and Stochastic Systems: }Exploring hybrid impulsive systems or systems under stochastic influences to address real-world uncertainties.

\textbf{Nonlinear and Non-Monotone Dynamics:} Studying nonlinear impulsive systems with multi-valued or non-monotone operators, including applications in solid mechanics and biological systems.

\textbf{Optimization Techniques:} Developing numerical and analytical methods to improve controllability in impulsive systems with constraints or limited resources.

\textbf{Applications in Control Engineering:} Applying theoretical results to practical scenarios in robotics, network control, and bio-inspired systems.

\textbf{Measure of Non-Compactness: }Exploring systems where the measure of non-compactness plays a role in characterizing approximate controllability.

Such investigations would enhance the understanding and application of impulsive systems in various fields.


\begin{thebibliography}{9}
	
\bibitem {zabczyk} Zabczyk, J. (2020). Mathematical control theory. Springer International Publishing.

\bibitem {laks} Lakshmikantham, V., \& Simeonov, P. S. (1989). Theory of impulsive differential equations (Vol. 6). World scientific.

\bibitem {pandit}Pandit, S. G., \& Deo, S. G. (2006). Differential systems involving impulses (Vol. 954). Springer.

\bibitem {ben} Benzaid, Z., \& Sznaier, M. (1993, June). Constrained controllability of linear impulse differential systems. In 1993 American Control Conference (pp. 216-220). IEEE.

\bibitem {geo} George, R. K., Nandakumaran, A. K., \& Arapostathis, A. (2000). A note on controllability of impulsive systems. Journal of Mathematical Analysis and Applications, 241(2), 276-283.

\bibitem {guan1}Guan, Z. H., Qian, T. H., \& Yu, X. (2002). On controllability and observability for a class of impulsive systems. Systems \& Control Letters, 47(3), 247-257.

\bibitem {guan2}Guan, Z. H., Qian, T. H., \& Yu, X. (2002). Controllability and observability of linear time-varying impulsive systems. IEEE Transactions on Circuits and Systems I: Fundamental Theory and Applications, 49(8), 1198-1208.

\bibitem {leela}Leela, S., McRae, F. A., \& Sivasundaram, S. (1993). Controllability of impulsive differential equations. Journal of Mathematical Analysis and Applications, 177(1), 24-30.

\bibitem {xie}Xie, G., \& Wang, L. (2004). Necessary and sufficient conditions for controllability and observability of switched impulsive control systems. IEEE Transactions on Automatic Control, 49(6), 960-966.

\bibitem {xie2}Xie, G., \& Wang, L. (2005). Controllability and observability of a class of linear impulsive systems. Journal of Mathematical Analysis and Applications, 304(1), 336-355.
\bibitem {han}Han, J., Liu, Y., Zhao, S., \& Yang, R. (2013). A note on the controllability and observability for piecewise linear time‐varying impulsive systems. Asian Journal of Control, 15(6), 1867-1870.

\bibitem {zhao1}Zhao, S., \& Sun, J. (2009). Controllability and observability for a class of time-varying impulsive systems. Nonlinear Analysis: Real World Applications, 10(3), 1370-1380.

\bibitem {zhao}Zhao, S., \& Sun, J. (2010). Controllability and observability for impulsive systems in complex fields. Nonlinear Analysis: Real World Applications, 11(3), 1513-1521.



\bibitem {muni}S Muni, V., \& K George, R. (2020). Controllability of linear impulsive systems–an eigenvalue approach. Kybernetika, 56(4), 727-752.



\bibitem {basmah}Bashirov, A. E., \& Mahmudov, N. I. (1999). On concepts of controllability for deterministic and stochastic systems. SIAM Journal on Control and Optimization, 37(6), 1808-1821.

\bibitem {mah}Mahmudov, N. I. (2003). ((Approximate controllability of semilinear deterministic and stochastic evolution equations in abstract spaces. SIAM journal on control and optimization, 42(5), 1604-1622.

\bibitem {mah1}Sakthivel, R., Ren, Y., \& Mahmudov, N. I. (2011). On the Approximate controllability of semilinear fractional differential systems. Computers \& Mathematics with Applications, 62(3), 1451-1459.






	\bibitem{1} Mahmudov, N. I. (2024). A study on Approximate controllability of linear impulsive equations in Hilbert spaces. Quaestiones Mathematicae, 1-16.
   \bibitem{3} Bainov, D., \& Simeonov, P. (2017). Impulsive differential equations: periodic solutions and applications. Routledge.
	\bibitem{4} Asadzade, J.A., \& Mahmudov, N.I. (2024). Approximate controllability of Linear Fractional Impulsive Evolution Equations in Hilbert Spaces. arXiv preprint  arXiv:2406.15114

	

   \bibitem{8}Li, X., Yong, J., Li, X., \& Yong, J. (1995). Control Problems in Infinite Dimensions. Optimal Control Theory for Infinite Dimensional Systems, 1-23.

  \bibitem{9} Wei, W., Xiang, X., \& Peng, Y. (2006). Nonlinear impulsive integro-differential equations of mixed type and optimal controls. Optimization, 55(1-2), 141-156.



   \bibitem{13} Leiva, H. (2015). Approximate controllability of semilinear impulsive evolution equations. In Abstract and Applied Analysis (Vol. 2015, No. 1, p. 797439). Hindawi Publishing Corporation.

   \bibitem{14} Sakthivel, R., \& Anandhi, E. R. (2010). Approximate controllability of impulsive differential equations with state-dependent delay. International Journal of Control, 83(2), 387-393.

   \bibitem{23} Grudzka, A., \& Rykaczewski, K. (2015). On Approximate controllability of functional impulsive evolution inclusions in a Hilbert space. Journal of Optimization Theory and Applications, 166, 414-439.

    \bibitem{54}  Vijayakumar, V. (2018). Approximate controllability results for impulsive neutral differential inclusions of Sobolev-type with infinite delay. International Journal of Control, 91(10), 2366-2386.

\bibitem{54a}  Jeet, K., \& Sukavanam, N. (2020). Approximate controllability of nonlocal and impulsive neutral integro-differential equations using the resolvent operator theory and an approximating technique. Applied Mathematics and Computation, 364, 124690.

\bibitem{54b} Shukla, A., Vijayakumar, V., \& Nisar, K. S. (2022). A new exploration on the existence and ((Approximate controllability for fractional semilinear impulsive control systems of order $r\in(1, 2)$. Chaos, Solitons \& Fractals, 154, 111615.

\end{thebibliography}
\end{document}